\documentclass[a4paper,10pt]{article}
\usepackage[margin=1.25in]{geometry}
\usepackage[]{algorithm2e}
\usepackage{amsthm}
\usepackage{caption}                         
\usepackage[utf8]{inputenc}
\usepackage{amssymb}
\usepackage{verbatim}
\usepackage{graphicx}
\usepackage{epstopdf}
\usepackage{indentfirst}                     
\usepackage[english]{babel}
\usepackage{subfig}                          
\usepackage{chngpage, calc}                  
\usepackage[font=small]{quoting}             
\usepackage[autostyle]{csquotes}             
\usepackage{amsmath}
\usepackage{mathtools}
\usepackage{bookmark}
\usepackage{mdframed}
\usepackage{xcolor}
\usepackage{enumerate}
\usepackage{hyperref}                        
\usepackage{algorithmic}
\usepackage[style=numeric,backend=bibtex]{biblatex}
\usepackage{empheq}
\usepackage{tikz}

\graphicspath{{Figures/}}
\bibliography{bibliography}
\numberwithin{equation}{section}

\newcommand{\dx}{dx}
\newcommand{\EM}{M}

\newcommand{\uht}{u_{h,\tau}^{(\mathbf{k})}}
\newcommand{\wht}{w_{h,\tau}^{(\mathbf{k})}}
\newcommand{\uhti}{u_{h,\tau}}
\newcommand{\whti}{w_{h,\tau}}
\newcommand{\uhkn}{u_{h,K_n}^n}
\newcommand{\whkn}{w_{h,K_n}^n}
\newcommand{\uhkmn}{u_{h,K_n-1}^n}
\newcommand{\whkmn}{w_{h,K_n-1}^n}
\newcommand{\uhnm}{u_{h,K_{n-1}}^{n-1}}
\newcommand{\whnm}{w_{h,K_{n-1}}^{n-1}}

\newcommand{\Tau}{\mathcal{T}}

\newcommand {\half}{ \frac{1}{2} }
\newcommand{\R}{\mathbb{R}}
\newcommand{\norm}[2]{{\left\| #2 \right\|}_{#1}}

\newcommand{\normal}{\nu}
\newcommand{\defeq}{\mathrel{\mathop:}=}

\theoremstyle{definition} 

\newtheorem{remark}{Remark}[section]

\theoremstyle{plain}
\newtheorem{theorem}{Theorem}[section]
\newtheorem*{theorem*}{Theorem}
\newtheorem{proposition}{Proposition}[section]
\newtheorem*{proposition*}{Proposition}
\newtheorem{lemma}{Lemma}[section]

\bibliography{references}

\usepackage{empheq}
\usepackage{tikz}
\graphicspath{{Figures/}}

\title{\textit{A posteriori} error estimates for the monodomain model in cardiac electrophysiology}
\author{
  Luca Ratti\footnotemark[1], \ 
	Marco Verani\footnotemark[1]
}

\date{January 22, 2019}

\definecolor{corr}{rgb}{0,0,0}
\definecolor{corr2}{rgb}{0,0,0}
\definecolor{corr3}{rgb}{1,0,0}

\begin{document}
\maketitle
\renewcommand{\thefootnote}{\fnsymbol{footnote}}
\footnotetext[1]{MOX-Dipartimento di Matematica, Politecnico di Milano, P.zza Leonardo da Vinci 32, 20133 Milano, Italy, \texttt{ \{luca.ratti,marco.verani\}@polimi.it}. The Authors are supported by INdAM-GNAMPA and INdAM-GNCS.  }
\begin{abstract}
We consider the monodomain model, a system of a parabolic semilinear reaction-diffusion equation coupled with a nonlinear ordinary differential equation, arising from the (simplified) mathematical description of the electrical activity of the heart. We derive \textit{a posteriori} error estimators accounting for different sources of error (space/time discretization and linearization). We prove reliability and efficiency (this latter under a suitable assumption) of the error indicators. Finally, numerical experiments assess the validity of the theoretical results. 
\end{abstract}

\section{Introduction}
The main goal of this paper is the \textit{a posteriori} numerical analysis of the monodomain model, a system of a parabolic semilinear reaction-diffusion equation coupled with a nonlinear ordinary differential equation, arising from the mathematical description of the electrical activity of the heart.
\textcolor{corr2}{The monodomain model represents a simplified version of the more realistic bidomain model which has been object in recent years of an intense research activity, see e.g. \cite{book:pavarino} and references therein. For the purpose of the paper, we first recall}
\cite{sanfelici2002convergence}, \textcolor{corr2}{where} a careful \textit{a priori} analysis of the Galerkin semidiscrete space approximation of the bidomain system is performed, investigating convergence properties and stability estimates for the semidiscrete solution. This result, coupled with the argument regarding the time-discretization analysis provided in \cite{art:collisavare}, allows for an exhaustive \textit{a priori} error analysis for the bidomain model. Moreover, in \cite{art:franzone} the authors introduce a space-time adaptive algorithm for the solution of the bidomain model by resorting to a stepsize control for the temporal adaptivity, whereas spatial adaptivity is performed by virtue of \textit{a posteriori} local error estimators. However, a complete \textit{a posteriori} error analysis is missing. 
\par
\textcolor{corr2}{With the aim of contributing to fill this gap, in this paper we focus on the simpler monodomain model and provide a detailed \textit{a posteriori} analysis.
In particular, we consider a Newton-Galerkin approximation of the monodomain system and look for \textit{a posteriori} indicators of the error involving the $L^2(0,T;H^1(\Omega))$ norm.} Inspired by the seminal work \cite{art:verfurth} and by the recent papers \cite{art:vohralik,art:amreinwihler}, we derive \textit{a posteriori} error bounds by providing a suitable splitting of the total residual into three operators, accounting for different sources of error entailed by the discretization process. Specifically, we introduce a linearization residual, a time discretization residual, and a space discretization residual, with the additional difficulty with respect, e.g., to \cite{art:amreinwihler} represented by the coupled structure of the system of differential equations.
\par
\textcolor{corr2}{ The \textit{a posteriori} analysis is complemented with an \textit{a priori} analysis which relies on previous results obtained in \cite{sanfelici1996}, where error estimates with respect to the $L^\infty(0,T;L^2(\Omega))$ norm of the error are obtained. Here, we derive \textit{a priori} estimates for the semidiscrete problem in a different norm involving the $L^2(0,T;H^1(\Omega))$ one.}
\par
\textcolor{corr2}{ The \textit{a posteriori} error estimators obtained in this paper can be employed to derive fully space-time adaptive algorithms that can be of particular importance, for instance, in the solution of inverse problems like the identification of ischemic regions}
(i.e. areas in which the coefficient of the system are altered from the reference values) by means of boundary voltage.
An iterative algorithm (as the one proposed in \cite{art:BRV} for a simplified model) would greatly benefit from an adaptive approach that would drastically reduce the computational cost.
\par
The paper is organized as follows: in Section \ref{approx} we introduce the Newton-Galerkin full discretization of the monodomain model, whereas Section \ref{apriori} is devoted to the \textit{a priori} estimates for the problem. In Section \ref{residual} we introduce the residual operators associated to the discrete solution and prove the equivalence between the error and the residual (in suitable norms). In Section \ref{estimators} we define three \textit{a posteriori} estimators and employ them to prove an upper bound for the approximation error. We also provide, under a suitable assumption, a lower estimate for the error in terms of the same indicators, thus assessing their efficiency. Finally, Section \ref{results} reports some numerical experiments assessing the validity of the derived estimates and investigating convergence rates both of the error and of the estimators as the discretization parameters are reduced.
\par

\section{A Newton-Galerkin scheme for the approximation of the monodomain model}
\sectionmark{A Newton-Galerkin scheme}
\label{approx}
Let $\Omega\subset \R^d$, $d=2,3$, be an open bounded domain. Consider the monodomain model (see \cite{book:pavarino,book:sundes-lines})
\begin{equation}
\left\{
\begin{aligned}
\partial_t u - \nabla \cdot(\EM \nabla u) + f(u,w) &= 0 \qquad &\text{in } \Omega \times (0,T), \\
\EM \partial_\normal u &= 0 \qquad &\text{on }  \partial \Omega \times (0,T), \\
u|_{t = 0} &= u_0 \qquad &\text{in }  \Omega, \\
\partial_t w + g(u,w) &= 0 \qquad &\text{in }  \Omega \times (0,T), \\
w|_{t = 0} &= w_0 \qquad &\text{in }  \Omega,
\end{aligned}
\right.
\label{eq:formaforte}
\end{equation}
being $u$ the trasmembrane electrical potential in the cardiac tissue and $\EM: \Omega \rightarrow \R^{d\times d}$ the conductivity tensor. In particular, according to the biological application, we assume that $\EM$ is constant in time, and in each point $x \in \Omega$ the tensor $\EM(x)$ is a symmetric positive definite matrix, with positive eigenvalues $\mu_i$, $i = 1, \ldots, d$. Moreover, we suppose that $\mu_i(x)$ are uniform in space and denote by $\mu_{min}$ and $\mu_{max}$ the minimum and the maximum eigenvalue, respectively. The associated eigenvectors may instead vary in space, and we assume that the overall matrix function $\EM(x)$ is smooth.
The nonlinear term $f(u,w)$ models the current induced by the motion of ions across the membrane, and is addressed as ionic current. According to a well established phenomenological approach (see, e.g., \cite{book:sundes-lines}), $f$ is a function of the potential $u$ and of a recovery variable $w$, whose dynamics is governed by a coupled nonlinear ordinary differential equation involving a nonlinear term $g$.  We focus in particular on the Aliev-Panfilov model of the cardiac tissue, according to the version reported, e.g., in \cite{art:BCP}; namely, the nonlinear terms $f$ and $g$ are as follows:
\begin{equation}
f(u,w) = A u(u-a)(u-1) + uw, \qquad g(u,w) = \epsilon (A u(u-1-a) + w),
\label{eq:AP}
\end{equation}
with $A,\epsilon > 0$, $0<a<1$.
\textcolor{corr2}{Such a problem is showed to be well-posed: in particular, we refer to \cite{art:BCR}, which extends the results contained in \cite{sanfelici1996} to the model of interest, and guarantees the following existence, uniqueness and comparison result:
\begin{proposition}
	Let the initial data $u_0 \in C^{2+\alpha}(\bar{\Omega})$, $w_0 \in C^{\alpha}(\bar{\Omega})$ satisfy the bound $0 \leq u_0 \leq 1$ and $0 \leq w_0 \leq \frac{A(1+a)^2}{4}$, consider $\EM \in C^2(\Omega)$ and let the following compatibility conditions hold: $\EM \nabla u_0 \cdot \normal = 0$, being $\partial \Omega \in C^{2+\alpha}$.	Then, there exists a unique classical solution $(u,w)$ of \eqref{eq:formaforte}, $u \in C^{2+\alpha,1+\alpha/2}(\overline{\Omega} \times [0,T])$ and $w \in C^{\alpha, 1+\alpha/2}(\overline{\Omega} \times [0,T])$. Moreover, it holds that
	\[
		0 \leq u(x,t) \leq 1, \quad 0 \leq w(x,t) \leq \frac{A(1+a)^2}{4} \qquad \forall (x,t) \in \Omega\times(0,T).
	\]
\label{prop:prel}
\end{proposition}
\begin{remark}
When considering $w_0 \in C^{2+\alpha}(\bar{\Omega})$, one can easily conclude (see \cite{art:BCR}) that also  $w \in C^{2+\alpha,1+\alpha/2}(\overline{\Omega} \times [0,T])$. In particular, both $u(\cdot,t)$ and $w(\cdot,t)$ belong to the Sobolev's space $H^2(\Omega)$ for each $t \in [0,T]$.
\end{remark}
\begin{remark}
If the conductivity tensor $\EM$ is only in $L^\infty(\Omega)$ (as in the case of an ischemic heart), we can nevertheless show (\cite{art:BCR}) the existence and uniqueness of the weak solution $(u,w)$ s.t. $u \in L^2(0,T;H^1(\Omega)) \cap L^\infty(0,T;L^2(\Omega))$, $\partial_t u \in L^2(0,T;H^*)$, $w \in L^{\infty}(0,T;L^2(\Omega))$, $\partial_t w \in L^2((0,T)\times\Omega)$, where $H^*=(H^1(\Omega))^*$. Moreover, the same bounds on $u,w$ hold as above and it is possible to guarantee additional regularity on the solution, namely $u \in C^{\alpha,\alpha/2}(\overline{\Omega} \times [0,T])$, $w \in C^{\alpha, 1+\alpha/2}(\overline{\Omega} \times [0,T])$. 
\label{rem:H2}
\end{remark}
}
\par 
The weak formulation of \eqref{eq:formaforte} reads
\begin{equation}
\left\{
\begin{aligned}
	\int_\Omega \partial_t u \varphi \dx + \int_{\Omega} \EM \nabla u \cdot \nabla \varphi \dx+ \int_\Omega f(u,w)\varphi \dx &= 0 \qquad \forall \varphi\in H^1(\Omega), \\
	\int_{\Omega}\partial_t w \psi \dx + \int_\Omega g(u,w) \psi \dx &= 0 \qquad \forall \psi \in L^2(\Omega).
\end{aligned}
\right.
\label{eq11:monodomain}
\end{equation}
For each time interval $(t_a,t_b)\subset (0,T)$, we introduce the following functional spaces:
\[ 
\begin{aligned}
 X(t_a,t_b) &= \{ u \textit{ s.t. } u \in L^2(t_a,t_b;H^1(\Omega)) \cap L^\infty(t_a,t_b;L^2(\Omega)), \partial_t u \in L^2(t_a,t_b;H^*)\} \\
 Y(t_a,t_b) &= \{ w \textit{ s.t. } w \in L^\infty(t_a,t_b;L^2(\Omega)), \partial_t w \in L^2((t_a,t_b)\times\Omega)\},
\end{aligned}
\]
which are Banach spaces endowed with the norms:
\[ 
\begin{aligned}
 \norm{X(t_a,t_b)}{u} &= \left( \norm{L^2(t_a,t_b;H^1(\Omega))}{u}^2 + \norm{L^\infty(t_a,t_b;L^2(\Omega))}{u}^2 + \norm{L^2(t_a,t_b;H^*)}{\partial_t u}^2\right)^{\half} \\
 \norm{Y(t_a,t_b)}{w} &= \left( \norm{L^\infty(t_a,t_b;L^2(\Omega))}{w}^2 + \norm{L^2((t_a,t_b)\times L^2(\Omega))}{\partial_t w}^2 \right)^{\half}.
\end{aligned}
\]
To ease the notation, we denote with $X$ and $Y$ the spaces $X(0,T)$ and $Y(0,T)$, respectively.   
\par
We now introduce a time semidiscretization of the problem by employing an implicit Euler scheme: consider a partition of the time interval
\[
\{t_n\}_{n=0}^N \subset [0,T]; \quad t_0 = 0, \ t_N = T; \quad t_{n} - t_{n-1} = \tau_n > 0,
\]
and define the semidiscrete solution as the couple of collections $(\{u^n\}_{n=0}^N,\{w^n\}_{n=0}^N)$, being $u^n \in H^1(\Omega)$, $w^n \in L^2(\Omega)$ with $n=0,\ldots,N$ such that
\begin{empheq}[left = \empheqlbrace]{align}
	&\qquad u^0 = u_0; \quad w^0 = w_0;  &\\
	&\int_\Omega \frac{u^n - u^{n-1}}{\tau^n}\varphi \dx + \int_\Omega \EM \nabla u^n \cdot \nabla \varphi \dx + \int_\Omega f(u^n,w^n)\varphi \dx = 0 \qquad &\forall \varphi \in H^1(\Omega),\label{eq:IE1}\\
	&\int_\Omega \frac{w^n - w^{n-1}}{\tau^n}\psi \dx + \int_\Omega g(u^n,w^n)\psi \dx = 0 \qquad &\forall \psi \in L^2(\Omega). \label{eq:IE2}
\end{empheq}
Consider the operators $\mathcal{F}^1: H^1(\Omega)\times L^2(\Omega) \rightarrow (H^1(\Omega))^*$, $\mathcal{F}^2: H^1(\Omega)\times L^2(\Omega) \rightarrow L^2(\Omega)$, which are defined interval-wise as follows: for $t \in (t_{n-1},t_n]$
\[
\begin{aligned}
\langle \mathcal{F}^1(u,w), \varphi \rangle &=\int_\Omega \frac{u - u^{n-1}}{\tau^n}\varphi \dx + \int_\Omega \EM\nabla u \cdot \nabla \varphi \dx + \int_\Omega f(u,w)\varphi \dx \\
\langle \mathcal{F}^2(u,w),\psi \rangle &=\int_\Omega \frac{w - w^{n-1}}{\tau^n}\psi \dx + \int_\Omega g(u,w)\psi \dx.
\end{aligned}
\]
The functionals $\mathcal{F}^1$ and $\mathcal{F}^2$ are (Fréchet) differentiable with respect to the $H^1(\Omega)$ norm in the variable $u$ and with respect to the $L^2(\Omega)$ norm in the variable $w$, respectively. This allows to define a Newton scheme for the solution of the nonlinear system \eqref{eq:IE1}-\eqref{eq:IE2} as follows: 
\par
\begin{algorithm}[H]
\label{al:Newt}
\begin{algorithmic}[1]
\STATE{ Set $u_0^n = u^{n-1}$, $w_0^n = u^{n-1}$,  $k = 1$\; }
\WHILE{ \textit{exit criterion is not satisfyed}}
\STATE{compute $\delta u$, $\delta w$ by solving
\begin{equation}
\begin{bmatrix}
\mathcal{F}^1_u(u_{k-1}^n,w_{k-1}^n) & \mathcal{F}^1_w(u_{k-1}^n,w_{k-1}^n)\\
\mathcal{F}^2_u(u_{k-1}^n,w_{k-1}^n) & \mathcal{F}^2_w(u_{k-1}^n,w_{k-1}^n)
\end{bmatrix}
\begin{bmatrix}
\delta u\\
\delta w
\end{bmatrix}
=
\begin{bmatrix}
-\mathcal{F}^1(u_{k-1}^n,w_{k-1}^n)\\
-\mathcal{F}^2(u_{k-1}^n,w_{k-1}^n)
\end{bmatrix} \textit{ in } H^* \times L^2(\Omega);
\label{eq:Newt}
\end{equation}
}
\STATE{update: $u^n_{k} = u^n_{k-1} + \delta u$, $w^n_{k} = w^n_{k-1} + \delta w$, $k = k+1$\;}
\ENDWHILE
\RETURN $u^{n} = u^n_{k}$, $w^{n} = w^n_{k}$.
\end{algorithmic}
\end{algorithm}
Computing the expression of the derivatives of $\mathcal{F}^1$ and $\mathcal{F}^2$, and substituting $\delta u = u_{k}^n - u_{k-1}^n$, $\delta w = w_{k}^n - w_{k-1}^n$, the system \eqref{eq:Newt} can be rewritten as
\begin{equation}
\begin{aligned}
&\int_\Omega \frac{u_{k}^n}{\tau^n}  \varphi \dx + \int_\Omega M\nabla u_{k}^n\cdot \nabla \varphi \dx+ \int_\Omega  \left[f_u(u_{k-1}^n,w_{k-1}^n)u_{k}^n + f_w(u_{k-1}^n,w_{k-1}^n)w_{k}^n \right]\varphi \dx = \int_\Omega \frac{u^{n-1}}{\tau^n}\varphi \dx\\
& \quad + \int_\Omega  \left[ f_u(u_{k-1}^n,w_{k-1}^n)u_{k-1}^n + f_w(u_{k-1}^n,w_{k-1}^n)w_{k-1}^n - f(u_{k-1}^n,w_{k-1}^n)\right]\varphi\dx  \quad \forall \varphi \in H^1(\Omega)
\end{aligned}
\label{eq:Newt1}
\end{equation}
\begin{equation}
\begin{aligned}
&\int_\Omega \frac{ w_{k}^n}{\tau^n}\psi \dx + \int_\Omega \left[ g_u(u_{k-1}^n,u_{k-1}^n)u_{k}^n + g_w(u_{k-1}^n,u_{k-1}^n)w_{k}^n \right]\psi \dx = \int_\Omega \frac{w^{n-1}}{\tau^n}\psi \dx
\\& \quad + \int_\Omega \left[ g_u(u_{k-1}^n,w_{k-1}^n)u_{k-1}^n + g_w(u_{k-1}^n,w_{k-1}^n)w_{k-1}^n - g(u_{k-1}^n,w_{k-1}^n)\right]\psi \dx \quad \forall \psi \in L^2(\Omega).
\end{aligned}
\label{eq:Newt2}
\end{equation}
\par
\textcolor{corr2}{Following \cite{art:verfurth}, we introduce an affinely equivalent, admissible, and shape-regular tessellation $\Tau_h^n$ for each instant $t_n$. For each element $K$ of $\Tau_h^n$, we denote by $h_K$ its diameter, and require $h_K \leq h$. We moreover require the following conditions to hold:}
\begin{enumerate}[i)]
	\item $\forall n \leq 1$, there exists a common refinement $\widetilde{\mathcal{T}}_h^n$ of both $\mathcal{T}_h^n$ and $\mathcal{T}_h^{n-1}$; 
	\item $\exists \rho_*, \rho^*>0$ independent of $n$ and $h$ s.t., defined 
	\[
	\rho(K',K) = \left\{\frac{h_{K'}}{h_K},\ K' \in \mathcal{T}_h^n, \ K \in \tilde{\mathcal{T}}_h^n:\ K\subset K'\right\},
	\] 
then $\rho_* \leq \rho(K',K) \leq \rho^*$ $\forall K \in \mathcal{T}_h^n$, $\forall n=1,\cdots,N$;
\end{enumerate}
Taking advantage of $\widetilde{\mathcal{T}}_h^n$, we introduce the Finite Element discrete space $V_h^n \subset H^1(\Omega)$
\[
V_h^n = \{ v_h \in C(\bar{\Omega}), v_h|_K \in \mathbb{P}_1(K) \text{ } \forall K \in \widetilde{\mathcal{T}}_h^n \}
\]
and the $L^2$ orthogonal projection $\Pi_h^n: L^2(\Omega) \rightarrow V_h^n$.
\par
The fully discrete solution of \eqref{eq:formaforte} consists in the pair of collections $(\{u_{h,k}^n\},\{u_{h,k}^n\})$, with $n = 0,\ldots,N$ and $k = 0,\ldots,K_n$, being $K_n$ the maximum number of iterations performed in each timestep (possibly varying with $n$). In particular, $\{\uhkn\}$ and $\{\whkn\}$ are such that:
\begin{itemize}
	\item $u_{h}^0 = \Pi_h^0 u_0$, $w_{h}^0 = \Pi_h^0 w_0$, the projections of the initial data on $\widetilde{\mathcal{T}}_h^0 = \mathcal{T}_h^0$ ;
	\item for each $n = 1, \cdots, N$, we initialize the Newton algorithm with $u_{h,0}^n = \Pi_h^n \uhnm$;
	\item for each $n = 1, \cdots, N$, for each $k = 1, \cdots, K_n$, the couple  $(u_{h,k}^n,w_{h,k}^n) \in V_h^n \times V_h^n$ solves the system \eqref{eq:Newt1}-\eqref{eq:Newt2} for all $(\varphi_h, \psi_h) \in V_h^n \times V_h^n$.
\end{itemize}

\section{A priori estimates for the space semidiscretization}
\label{apriori}
In this section we consider \textit{a priori} error estimates for the space semidiscretized problem under the assumption that the same tessellation $\Tau_h$ is considered in each instant, together with the discrete space $V_h$ of linear finite elements. We refer to the space semidiscrete solution as to the couple of functions $(u_h,w_h): [0,T]\rightarrow V_h \times V_h$ satisfying $u_h(0)=u_h^0$, $w_h(0)=w_h^0$ and
\begin{equation}
\left\{
\begin{aligned}
	\int_\Omega \partial_t u_h \varphi_h \dx + \int_{\Omega} \EM \nabla u_h \cdot \nabla \varphi_h \dx + \int_\Omega f(u_h,w_h)\varphi_h \dx &= 0 \qquad \forall \varphi_h \in V_h, \\
	\int_{\Omega}\partial_t w_h \psi_h \dx+ \int_\Omega g(u_h,w_h) \psi_h \dx &= 0 \qquad \forall \psi_h \in V_h.
\end{aligned}
\right.
\label{eq11:h}
\end{equation}
Taking advantage of standard inverse estimates and approximation results (see \cite{book:brenner2007}), it is possible to prove the following result:
\begin{theorem}
There exists a unique solution $(u_h,w_h)$ of problem \eqref{eq11:h} in $(C^1(0,T;V_h))^2$. Moreover, for any fixed $h_0$ there exists a positive $\delta_0$ such that $(u_h,w_h) \in S_{\delta_0}$ $\forall x,t \in \Omega \times [0,T]$, being $S_{\delta_0} = [-\delta_0, 1+\delta_0]\times [-\delta_0,\frac{A(1+a)^2}{4}+\delta_0]$. Finally, there exists a constant $c$ depending on $u,w,u_0,w_0,f,g,\Omega,T$ and independent of $h$ such that
\begin{equation}
\left\{\norm{L^\infty(0,T;L^2(\Omega))}{u-u_h}^2 + \norm{L^\infty(0,T;L^2(\Omega))}{w-w_h}^2\right\}^\half \leq c h^2. 
\label{eq:apriori1}
\end{equation}
\label{th:apriori1}
\end{theorem}
\textcolor{corr2}{The proof of this theorem relies on techniques introduced in \cite{thomeewahlbin}: with minor modifications, it is possible to adapt the proof of \cite[Theorem 4.4]{sanfelici1996} to the present context where the Aliev-Panfilov electrophysiological model is considered.}
We are moreover interested in establishing the convergence rate of the $X$ and $Y$ norms of the error. This is the object of the following result:
\begin{theorem}
There exists a constant $c$ depending on $u,w,u_0,w_0,f,g,\Omega,T$ and independent of $h$ such that
\begin{equation}
\left\{\norm{X(0,T)}{u-u_h}^2 + \norm{Y(0,T)}{w-w_h}^2\right\}^\half \leq c h. 
\label{eq:apriori2}
\end{equation}
\label{th:apriori2}
\end{theorem}
\begin{proof}
Consider the equations of system \eqref{eq11:monodomain}, test them with the functions $u_h-\varphi_h$ and $w_h-\psi_h$ respectively, being $\varphi_h$, $\psi_h \in V_h$, and sum them. Repeating the same procedure on system \eqref{eq11:h} and subracting the two equations obtained, we get
\begin{equation}
\begin{aligned}
\half \frac{d}{dt} &\left( \norm{L^2(\Omega)}{u-u_h}^2 + \norm{L^2(\Omega)}{w-w_h}^2 \right) + \norm{L^2(\Omega)}{\nabla(u-u_h)}^2 = \int_\Omega \partial_t(u-u_h) (u-\varphi_h)\dx \\
&+ \int_\Omega \partial_t(w-w_h) (w-\psi_h)\dx + \int_\Omega \nabla(u-u_h)\cdot \nabla(u-\varphi_h)\dx \\
&+ \int_\Omega (f(u,w)-f(u_h,w_h)(u_h - \varphi_h)\dx + \int_\Omega (g(u,w)-g(u_h,w_h)(w_h - \psi_h)\dx.
\end{aligned}
\label{eq:Quadrato}
\end{equation}
Consider now $\varphi_h = \Pi_h u$ and $\psi_h = \Pi_h w$, being $\Pi_h$ the $L^2$ orthogonal projection on $V_h$ operator, and observe that 
\[
\begin{aligned}
\int_\Omega \partial_t(u-u_h) (u-\Pi_h u) dx &= \int_\Omega \partial_t(u-\Pi_h u) (u-\Pi_h u)dx + \int_\Omega \partial_t(\Pi_h u - u_h) (u-\Pi_h u) dx\\
&= \half \frac{d}{dt}\norm{L^2(\Omega)}{u-\Pi_h u}^2.
\end{aligned}
\]
A similar result hold for the second term on the rihgt-hand side of \eqref{eq:Quadrato}. By Cauchy-Schwarz and Young inequalities, we conclude that
\[
\begin{aligned}
\half \frac{d}{dt} &\left( \norm{L^2(\Omega)}{u-u_h}^2 + \norm{L^2(\Omega)}{w-w_h}^2 \right) + \half \norm{L^2(\Omega)}{\nabla (u-u_h)}^2 = \\
& \half \frac{d}{dt} \left( \norm{L^2(\Omega)}{u-\Pi_h u}^2 + \norm{L^2(\Omega)}{w-\Pi_h w}^2 \right) + \half \norm{L^2(\Omega)}{\nabla (u-\Pi_h u)}^2 + |E_{\Pi_h}|,
\end{aligned}
\]
being 
\[
E_{\Pi_h} = \int_\Omega (f(u,w)-f(u_h,w_h))(u_h - \Pi_h u)\dx + \int_\Omega (g(u,w)-g(u_h,w_h))(w_h - \Pi_h w)\dx.
\]
Integrating from $0$ to $t$, and employing the fundamental theorem of calculus, together with the choice $u_h(0) = u_{h}^0 = \Pi_h u_0$, $w_h(0) = w_{h}^0 = \Pi_h w_0$, we get
\begin{equation}
\begin{aligned}
&\norm{L^2(\Omega)}{u(t)-u_h(t)}^2 + \norm{L^2(\Omega)}{w(t)-w_h(t)}^2 + \int_0^t \norm{L^2(\Omega)}{\nabla (u(s)-u_h(s))}^2ds \\
& \quad \leq \norm{L^2(\Omega)}{u_0 - \Pi_h u_0}^2 + \norm{L^2(\Omega)}{w_0 - \Pi_h w_0}^2 + \norm{L^2(\Omega)}{u(t) - \Pi_h u(t)}^2 + \norm{L^2(\Omega)}{w(t) - \Pi_h w(t)}^2 \\
& \quad + \int_0^t \norm{L^2(\Omega)}{\nabla(u(s) - \Pi_h u(s))}^2 ds + 2 \int_0^t |E_{\Pi_h}(s)|ds.
\end{aligned}
\label{eq:aux}
\end{equation}
It immediately follows that 
\[
\begin{aligned}
&\norm{L^\infty(0,T;L^2(\Omega))}{u-u_h}^2 + \norm{L^\infty(0,T;L^2(\Omega))}{w-w_h}^2 \leq \norm{L^2(\Omega)}{u_0 - \Pi_h u_0}^2 + \norm{L^2(\Omega)}{w_0 - \Pi_h w_0}^2 \\ 
& \quad + \norm{L^\infty(0,T;L^2(\Omega))}{u-\Pi_h u}^2 + \norm{L^\infty(0,T;L^2(\Omega))}{w-\Pi_h w}^2 + \norm{L^2(0,T;H^1(\Omega))}{u-\Pi_h u}^2 \\
& \quad + 2 \int_0^T |E_{\Pi_h}(s)|ds.
\end{aligned}
\] 
Now, we observe that, since both $(u,w)$ and $(u_h,w_h)$ belong to $S_{\delta_0}$ for a suitable value of $\delta_0$ (see Theorem \ref{th:apriori1}) and since the functions $f,g$ are Lipschitz continuous on $S_{\delta_0}$ with constants bounded by $c_{\delta_0} >0$, it holds
\[
\begin{aligned}
\int_0^T |E_{\Pi_h}(s)|ds \leq & c_{\delta_0}  \int_0^T \int_\Omega (u-u_h)(u-\Pi_h u)+(w-w_h)(u-\Pi_h u)\dx ds \\
& \quad + c_{\delta_0} \int_0^T \int_\Omega (u-u_h)(w-\Pi_h w) + (w-w_h)(w-\Pi_h w)\dx ds \\
\leq & c_{\delta_0} \int_0^T \left( \norm{L^2(\Omega)}{u-\Pi_h u} + \norm{L^2(\Omega)}{w-\Pi_h w} + 2\int_\Omega (u-\Pi_h u)(w-\Pi_h w)\dx \right) ds \\
\leq & 2c_{\delta_0}T \left( \norm{L^\infty(0,T;L^2(\Omega))}{u-\Pi_h u}^2 + \norm{L^\infty(0,T;L^2(\Omega))}{w-\Pi_h w}^2 \right).
\end{aligned}
\]
In conclusion, we have
\[
\begin{aligned}
&\norm{L^\infty(0,T;L^2(\Omega))}{u-u_h}^2 + \norm{L^\infty(0,T;L^2(\Omega))}{w-w_h}^2 \lesssim \norm{L^2(\Omega)}{u_0 - \Pi_h u_0}^2 + \norm{L^2(\Omega)}{w_0 - \Pi_h w_0}^2 \\ 
& \quad + \norm{L^\infty(0,T;L^2(\Omega))}{u-\Pi_h u}^2 + \norm{L^\infty(0,T;L^2(\Omega))}{w-\Pi_h w}^2 + \norm{L^2(0,T;H^1(\Omega))}{u-\Pi_h u}^2.
\end{aligned}
\] 
\textcolor{corr2}{
Applying standard approximation properties of $V_h$, taking advantage of the fact that both $u(\cdot,t)$ and $w(\cdot,t)$ belong to $H^2(\Omega)$ for $t \in [0,T]$ (see Remark \ref{rem:H2}), we can conclude that the following suboptimal estimate holds:
\[
\left\{\norm{L^\infty(0,T;L^2(\Omega))}{u-u_h}^2 + \norm{L^\infty(0,T;L^2(\Omega))}{w-w_h}^2\right\}^\half \leq c h.
\] 
}
In view of this estimate, from \eqref{eq:aux} we infer that $\norm{L^2(0,T;H^1(\Omega))}{u-u_h} \leq c h$. 
\par 
\textcolor{corr2}{
To conclude, we need to consider the terms involving the derivative in time. 
This requires the introduction of the elliptic projection operator associated to the bilinear form $\int_\Omega M \nabla u \cdot \nabla v + \mu_{min} \int_\Omega uv$, i.e., the map $R_h: H^1(\Omega) \rightarrow V_h$ such that
\begin{equation}
\int_\Omega M (u -\nabla R_h u) \cdot \nabla \varphi_h \dx + \mu_{min} \int_\Omega (u - R_h u)\varphi_h \dx = 0 \qquad \forall \varphi_h \in V_h.
\label{eq:Rh}
\end{equation}
According to the properties of $R_h$ (see, e.g., \cite{book:thomee}), we know that $\forall u \in H^1(\Omega)$ it holds
\begin{equation}
	\norm{L^2(\Omega)}{u - R_h u} \leq h \norm{H^1(\Omega)}{u}.
\label{eq:nitsche}
\end{equation}	
By employing the first equation in system \eqref{eq11:h}, for each $\varphi \in H^1(\Omega)$ it holds
\[
\begin{aligned}
\langle \partial_t u_h, \varphi \rangle &= \int_\Omega \partial_t u_h (\varphi-R_h \varphi)\dx + \int_\Omega \partial_t u_h\ R_h \varphi \dx \\
&= \int_\Omega \partial_t u_h (\varphi-R_h \varphi)\dx - \int_\Omega \EM \nabla u_h \cdot \nabla R_h \varphi \dx - \int f(u_h,w_h)R_h \varphi\dx.
\end{aligned}
\]
According to \eqref{eq11:monodomain}, and in view of \eqref{eq:Rh}, we can conclude that $\forall \varphi \in H^1(\Omega)$
\[
\begin{aligned}
\langle \partial_t (u-u_h), \varphi \rangle = & - \int_\Omega \EM \nabla(u-u_h)\cdot \nabla \varphi \dx - \int_\Omega (f(u,w)-f(u_h,w_h)) \varphi\dx \\& - \int_\Omega \EM \nabla u_h \cdot \nabla (\varphi - R_h \varphi)\dx - \int_\Omega f(u_h,w_h)(\varphi - R_h \varphi)\dx - \int_\Omega \partial_t u_h(\varphi - R_h \varphi)\dx \\
= & - \int_\Omega \EM \nabla(u-u_h)\cdot \nabla \varphi \dx - \int_\Omega (f(u,w)-f(u_h,w_h)) \varphi \dx \\& + \mu_{min} \int_\Omega u_h (\varphi - R_h \varphi)\dx - \int_\Omega f(u_h,w_h)(\varphi - R_h \varphi)\dx - \int_\Omega \partial_t u_h(\varphi - R_h \varphi)\dx.
\end{aligned}
\]
}
{
Via Cauchy-Schwarz inequality we obtain
\[
\begin{aligned}
\langle \partial_t (u-u_h), \varphi \rangle \leq & \mu_{max} \norm{H^1(\Omega)}{u-u_h}\norm{H^1(\Omega)}{\varphi} + K_f (\norm{L^2(\Omega)}{u-u_h} + \norm{L^2(\Omega)}{w-w_h})\norm{L^2(\Omega)}{\varphi} \\
& + \mu_{min} \norm{L^2(\Omega)}{u_h} \norm{L^2(\Omega)}{\varphi - R_h \varphi} + \vert\Omega\vert^\half \norm{L^\infty(\Omega)}{f(u_h,w_h)}\norm{L^2(\Omega)}{\varphi - R_h \varphi} \\
& + \norm{L^2(\Omega)}{\partial_t u_h}\norm{L^2(\Omega)}{\varphi - R_h \varphi}.
\end{aligned}
\]
Now, we show that $\norm{L^2((0,T)\times\Omega)}{\partial_t u_h}$ is bounded by a constant independent of $h$. Indeed, considering $\varphi_h = \partial_t u_h$ in the first equation of \eqref{eq11:h}, we obtain
\[
\norm{L^2(\Omega)}{\partial_t u_h}^2 + \frac{d}{dt}\norm{L^2(\Omega)}{\sqrt{M}\nabla u_h}^2 + \int_\Omega f(u_h,w_h) \partial_t u_h \dx = 0.
\]
Integrating from $0$ to $T$, we get
\[
\norm{L^2((0,T)\times\Omega)}{\partial_t u_h}^2 + \norm{L^2(\Omega)}{\sqrt{M}\nabla u_h(\cdot,T)}^2 - \norm{L^2(\Omega)}{\sqrt{M}\nabla u_{h,0}}^2 + \int_0^T \int_\Omega f(u_h,w_h) \partial_t u_h \dx \ dt =0.
\]
Thus, it holds that
\[
\norm{L^2((0,T)\times\Omega)}{\partial_t u_h}^2 \leq \norm{L^2((0,T)\times\Omega)}{f(u_h,w_h)}\norm{L^2((0,T)\times\Omega)}{\partial_t u_h} + \mu_{max}\norm{H^1(\Omega)}{u_{h,0}}^2,
\]
and by solving the second-order inequality, we conclude that 
\begin{equation}
{\color{black}
\norm{L^2((0,T)\times\Omega)}{\partial_t u_h} \leq \half \left( \norm{L^2((0,T)\times\Omega)}{f(u_h,w_h)} + \left( \norm{L^2((0,T)\times\Omega)}{f(u_h,w_h)}^2 + 4\mu_{max}\norm{H^1(\Omega)}{u_{h,0}}^2 \right)^\half\right) \leq C.}
\label{eq:uhtL2}
\end{equation}
In view of \eqref{eq:uhtL2}, employing \eqref{eq:nitsche}, the above  estimate for $\norm{L^2(0,T;H^1(\Omega))}{u-u_h}$, together with the estimates for $\norm{L^{\infty}(0,T;L^2(\Omega))}{u-u_h}$ and $\norm{L^{\infty}(0,T;L^2(\Omega))}{w-w_h}$ in Theorem \ref{th:apriori1}, we get
\[
\norm{L^2(0,T;H^*)}{\partial_t (u - u_h)} \leq c h.
\] 
}
An analogous argument holds for $\norm{L^2((0,T)\times \Omega)}{\partial_t(w-w_h)}$, and the thesis follows.
\end{proof}

\begin{remark}
When stating the discrete problem \eqref{eq11:h}, we have neglected any error introduced by the computation of the integral $\int_\Omega M \nabla u_h \cdot \nabla w_h \dx$. When $M$ is a polynomial function, the integration can be performed exactly by choosing a suitable quadrature rule. In case $M$ is not a polynomial but still sufficiently smoot (e.g., $M \in C^{1+\alpha}(\Omega)$), the quadrature error do not affect the results contained in Theorems \ref{th:apriori1} and \ref{th:apriori2}, as can be verified by an application of Strang's lemma. When considering the case of a piecewse smooth coefficient $M$ (which occurs, e.g., when modeling an ischemic cardiac tissue), one should adopt a different strategy, as suggested, e.g., in \cite{chenzou}.
\end{remark}


\section{Residual operators}
\label{residual}
We now move towards the introduction of \textit{a posteriori} estimators. Consider the fully discrete solution $(\{u_{h,k}^n\},\{w_{h,k}^n\})$ as introduced in Section 1, being again $\{\Tau_h^n\}_n$ possibly different tessellations among the different discrete instants. Collecting all the final indices $K_n$ in a multi-index $\mathbf{k} = [K_n]_{n=1}^N$, the associated linear interpolated solution $(\uht,\wht)$ is a couple of continuous functions on $[0,T]$, defined timestep-wise as follows: for each $t \in (t_{n-1},t_n]$, $n  = 1,\ldots,N$,
\begin{equation}
\uht = \frac{t-t_{n-1}}{\tau_n}\uhkn + \frac{t_n - t}{\tau_n}\uhnm, \qquad \wht = \frac{t-t_{n-1}}{\tau_n}\whkn + \frac{t_n - t}{\tau_n}\whnm.
\label{eq:interpolants}
\end{equation}
We now define for almost each instant $t$ the \textit{residual} operator $R(t)$ in the product space $(H^1(\Omega) \times L^2(\Omega))^* = H^* \times L^2(\Omega)$, being $H^*$ the dual space of $H^1(\Omega)$:
\begin{equation}
\begin{aligned}
\langle R(t),(\varphi,\psi) \rangle = & \langle R_1(t),\varphi \rangle + \langle R_2(t),\psi \rangle \qquad \forall \varphi \in H^1(\Omega), \psi \in L^2(\Omega)\\
\langle R_1(t),\varphi \rangle = & - \int_\Omega \partial_t \uht \varphi \dx - \int_\Omega \EM \nabla \uht \cdot \nabla \varphi \dx - \int_\Omega f(\uht,\wht)\varphi \dx \\
\langle R_2(t),\psi \rangle = & - \int_\Omega \partial_t \wht \psi\dx - \int_\Omega g(\uht,\wht)\psi \dx.
\end{aligned}
\label{eq:res}
\end{equation} 
\par
It is now possible to prove a result of equivalence between the $X,Y$ norms of the error and the dual norms of the residual operators. More precisely, it holds:
\begin{theorem}
The functions $\norm{H^*}{R_1(t)}$ and $\norm{L^2(\Omega)}{R_2(t)}$ are square integrable on each interval $(t_a,t_b) \subset (0,T)$, and moreover
\begin{subequations}
\begin{equation}
\left\{ \norm{L^2(t_a,t_b,H^*)}{R_1}^2 + \norm{L^2((t_a,t_b)\times\Omega)}{R_2}^2 \right\}^\half \leq c^* \left\{ \norm{X(t_a,t_b)}{u - \uht}^2 + \norm{Y(t_a,t_b)}{w-\wht}^2 \right\}^\half \\
\label{eq:A}
\end{equation}
\begin{equation}
\begin{aligned}
c_* \left\{\norm{X(0,t)}{u - \uht}^2 + \norm{Y(0,t)}{w-\wht}^2 \right\}^\half  &\leq \left\{ \norm{L^2(\Omega)}{u_0 - \Pi_h^0 u_0}^2 + \norm{L^2(\Omega)}{w_0 - \Pi_h^0 w_0}^2 \right.\\
&\quad + \left. \norm{L^2(0,t,H^*)}{R_1}^2 + \norm{L^2((0,t)\times\Omega)}{R_2}^2\right\}^\half, 
\end{aligned}
\label{eq:B}
\end{equation}
\end{subequations}
where $c_*$ and $c^*$ depend on $\Omega,\mu_{max},\mu_{min},f,g$ and $T$.
\label{th11:1}
\end{theorem}
\begin{proof}
By employing equation \eqref{eq11:monodomain} together with the expressions of $R_1(t)$ and $R_2(t)$ we have, $\forall \varphi \in H^1(\Omega)$, $\forall \psi \in L^2(\Omega)$, $a.e.$ $t \in (0,T)$
\begin{equation}
\begin{aligned}
&	\int_\Omega \partial_t (u - \uht)\varphi \dx + \int_\Omega \EM\nabla (u-\uht)\cdot\nabla\varphi \dx + \int_\Omega (f(u,w)-f(\uht,\wht))\varphi \dx \\
&	+ \int_\Omega \partial_t(w-\wht)\psi \dx + \int_\Omega (g(u,w) - g(\uht,\wht))\psi \dx = \langle R_1(t), \varphi \rangle + \langle R_2(t), \psi \rangle.
\end{aligned}
\label{eq:1}
\end{equation}
Fixing $\psi = 0$ and employing the Cauchy-Schwarz inequality and 
the fact that $f$ is Lipschitz continuous with constant $K_f$, 
\[
\begin{aligned}
|\langle R_1(t), \varphi \rangle| \leq & \left( \norm{H^*}{\partial_t(u-\uht)} + \mu_{max}\norm{H^1(\Omega)}{u-\uht} + K_f \norm{L^2(\Omega)}{u-\uht} \right. \\
& \quad \left. + K_f \norm{L^2(\Omega)}{w-\wht}\right) \norm{H^1(\Omega)}{\varphi}.
\end{aligned}
\]
Thus, computing the $L^2$ norm on $(t_a,t_b)$ we obtain
\begin{equation}
\begin{aligned}
\norm{L^2(t_a,t_b; H^*)}{R_1} \leq & \left( \norm{L^2(t_a,t_b;H^*)}{\partial_t(u-\uht)} + \mu_{max} \norm{L^2(t_a,t_b;H^1)}{u-\uht} \right. \\&+ \left. K_f \norm{L^2((t_a,t_b)\times \Omega)}{u-\uht} + K_f \norm{L^2((t_a,t_b)\times \Omega)}{w-\wht}\right). 
\end{aligned}
\label{eq:R1}
\end{equation}
Analogously, when taking $\varphi = 0$, we get
\begin{equation}
\begin{aligned}
\norm{L^2((t_a,t_b)\times \Omega)}{R_2} \leq & \left( \norm{L^2((t_a,t_b)\times \Omega)}{\partial_t(w-\wht)} + K_g \norm{L^2((t_a,t_b)\times \Omega)}{u-\uht} \right. \\& + \left. K_g \norm{L^2((t_a,t_b)\times \Omega)}{w-\wht}\right), 
\end{aligned}
\label{eq:R2}
\end{equation}
being $K_g$ the Lipschitz constant of $g$. Summing \eqref{eq:R1} and \eqref{eq:R2} we obtain \eqref{eq:A}.
\par
To prove \eqref{eq:B}, consider \eqref{eq:1} and take $\varphi = u - \uht$, $\psi = w - \wht$; by mean value theorem \footnote[2]{ Applied to the real valued function $\textit{h}: \zeta \in \R \rightarrow h(\zeta) = \int_{\Omega}f(u+\zeta(\uht-u),w+\zeta(\wht-w))(\uht-u)\dx$, it guarantees that there exists $\zeta^*\in [0,1]$ $s.t.$ $\textit{h}(1)-\textit{h}(0)= h'(\zeta*)$.} it holds that $\exists (\xi_1,\eta_1) = (u,w) + \zeta_1(\uht-u,\wht-1), (\xi_2,\eta_2) = (u,w) + \zeta_2(\uht-u,\wht-1)$,  such that
\[
\begin{aligned}
\half \frac{d}{dt} &\left( \norm{L^2(\Omega)}{u-\uht}^2 + \norm{L^2(\Omega)}{w-\wht}^2 \right) + \int_\Omega \EM\nabla(u-\uht) \cdot \nabla(u-\uht) \dx \\
&+\int_\Omega (f_u(\xi_1,\eta_1)(u-\uht) + f_w(\xi_1,\eta_1)(w-\wht))(u-\uht)\dx \\
&+\int_\Omega (g_u(\xi_2,\eta_2)(u-\uht) + g_w(\xi_2,\eta_2)(w-\wht))(w-\wht)\dx\\
&= \langle R_1, u-\uht \rangle + \langle R_2, w-\wht \rangle.
\end{aligned} 
\]
Consider now the quadratic form $\mathcal{Q}:H^1(\Omega) \times L^2(\Omega) \rightarrow \R$,
\[
\mathcal{Q}(m,n) = \int_\Omega \left( -f_u(\xi_1,\eta_1)m^2 - \left( f_w(\xi_1,\eta_1) + g_u(\xi_2,\eta_2)\right)mn - g_w(\xi_2,\eta_2)n^2 \right) \dx,
\]
which allows to rewrite the previous equation as
\[
\begin{aligned}
\half \frac{d}{dt} &\left( \norm{L^2(\Omega)}{u-\uht}^2 + \norm{L^2(\Omega)}{w-\wht}^2 \right) + \int_\Omega \EM\nabla(u-\uht) \cdot \nabla(u-\uht) \dx \\
& = \mathcal{Q}(u-\uht,w-\wht) + \langle R_1, u-\uht \rangle + \langle R_2, w-\wht \rangle.
\end{aligned}
\]
It clearly holds that $|\mathcal{Q}(m,n)|\leq \lambda_{max} (\norm{L^2(\Omega)}{m} + \norm{L^2(\Omega)}{n})$, being $\lambda_{max}$ a continuous function of $f_u(\xi_1,\eta_1), f_w(\xi_1,\eta_1), g_u(\xi_2,\eta_2),g_w(\xi_2,\eta_2)$. Hence, $\lambda_{max}$ depends both on $x$ and $t$, but thanks to \textit{a priori} bounds on $(\xi_1,\eta_1)$ and $(\xi_2,\eta_2)$ (inherited from Proposition \ref{prop:prel} and Theorem \ref{th:apriori1}), we can ensure it is bounded from above on $\Omega \times (0,T)$ by a positive constant $\Lambda$. Via Cauchy-Schwarz and Young inequalities,
\[
\begin{aligned}
	&\half \frac{d}{dt} \left( \norm{L^2(\Omega)}{u-\uht}^2 + \norm{L^2(\Omega)}{w-\wht}^2 \right) + \mu_{min} \norm{H^1(\Omega)}{u-\uht}^2 \leq \mu_{min} \norm{L^2(\Omega)}{u-\uht}^2 \\ 
	&\quad + \Lambda \left( \norm{L^2(\Omega)}{u-\uht}^2 + \norm{L^2(\Omega)}{w-\wht}^2 \right) + \frac{1}{2\mu_{min}}\left( \norm{H^*}{R_1}^2 + \norm{L^2(\Omega)}{R_2}^2 \right) \\
	&\quad + \frac{\mu_{min}}{2}\left( \norm{H^1(\Omega)}{u-\uht}^2 + \norm{L^2(\Omega)}{w-\wht}^2 \right),
\end{aligned}
\]
hence
\[
\begin{aligned}
	&\half \frac{d}{dt} \left( \norm{L^2(\Omega)}{u-\uht}^2 + \norm{L^2(\Omega)}{w-\wht}^2 \right) + \frac{\mu_{min}}{2} \norm{H^1(\Omega)}{u-\uht}^2 \\ 
	&\quad \leq (\Lambda + \mu_{min}) \left( \norm{L^2(\Omega)}{u-\uht}^2 + \norm{L^2(\Omega)}{w-\wht}^2 \right) + \frac{1}{2\mu_{min}}\left( \norm{H^*}{R_1}^2 + \norm{L^2(\Omega)}{R_2}^2 \right).
\end{aligned}
\]
Let us now take a fixed $t \in (0,T)$ and integrate from $0$ to $t$, obtaining
\begin{equation}
\begin{aligned}
&\left( \norm{L^2(\Omega)}{(u-\uht)(t)}^2 + \norm{L^2(\Omega)}{(w-\wht)(t)}^2 \right) +  \mu_{min}\norm{L^2(0,t;H^1)}{u-\uht}^2 \leq \\
&\quad \int_0^t 2(\Lambda + \mu_{min}) \left( \norm{L^2(\Omega)}{(u-\uht)(s)}^2 + \norm{L^2(\Omega)}{(w-\wht)(s)}^2 \right)ds \\
&\quad + \int_0^t \frac{1}{\mu_{min}}\left( \norm{H^*}{R_1(s)}^2 + \norm{L^2(\Omega)}{R_2(s)}^2 \right)ds + \left( \norm{L^2(\Omega)}{u_0-\uht(0)}^2 + \norm{L^2(\Omega)}{w_0-\wht(0)}^2 \right).
\end{aligned}
\label{eq:belowt}
\end{equation}
Via Gronwall's inequality, we obtain
\[
\begin{aligned}
\left( \norm{L^2(\Omega)}{u(t)-\uht(t)}^2 + \norm{L^2(\Omega)}{w(t)-\wht(t)}^2 \right) \leq &e^{2(\Lambda + \mu_{min})t} \Bigl( \norm{L^2(\Omega)}{u_0 - \Pi_h^0 u_0}^2 + \norm{L^2(\Omega)}{w_0 - \Pi_h^0 w_0}^2  \\
&  + \frac{1}{\mu_{min}} \left( \norm{L^2(0,t;H^*)}{R_1}^2 + \norm{L^2((0,t)\times\Omega)}{R_2}^2 \right) \Bigr),
\end{aligned}
\]
whence the bound on $\norm{L^\infty(0,t,L^2(\Omega))}{u-\uht}$ and $\norm{L^\infty(0,t,L^2(\Omega))}{w-\wht}$. Moreover, from \eqref{eq:belowt} we get
\[
\begin{aligned}
& \mu_{min} \norm{L^2(0,t;H^1)}{u-\uht}^2 \leq 2(\Lambda + \mu_{min})t \left(\norm{L^\infty(0,t,L^2(\Omega))}{u-\uht}^2+\norm{L^\infty(0,t,L^2(\Omega))}{w-\wht}^2 \right) \\
&\quad + \frac{1}{\mu_{min}} \left( \norm{L^2(0,t;H^*)}{R_1}^2 + \norm{L^2((0,t)\times\Omega)}{R_2}^2 \right) + \left( \norm{L^2(\Omega)}{u_0-\uht(0)}^2 + \norm{L^2(\Omega)}{w_0-\wht(0)}^2 \right) .
\end{aligned}
\]
Finally, taking $\psi = 0$ in \eqref{eq:1}, by Cauchy-Schwarz inequality we get
\[
\begin{aligned}
\norm{H^*}{\partial_t (u - \uht)(t)} \leq &\mu_{max} \norm{H^1(\Omega)}{(u-\uht(t))} + K_f \norm{L^2(\Omega)}{(u-\uht)(t)} \\
&+ K_f \norm{L^2(\Omega)}{(w-\wht)(t)} + \norm{H^*}{R_1},
\end{aligned}
\]
thus
\[
\begin{aligned}
\norm{L^2((0,t) \times \Omega)}{\partial_t (u - \uht)} \leq & \mu_{max} \norm{L^2(0,t;H^1(\Omega))}{u-\uht} + K_f \sqrt{t} \norm{L^\infty(0,t;L^2(\Omega))}{u-\uht} \\&+ K_f \sqrt{t} \norm{L^\infty(0,t;L^2(\Omega))}{w-\wht} + \norm{L^2(0,t;H^*)}{R_1}.
\end{aligned}
\]
A similar strategy allows to conclude that an analogous bound holds for $\partial_t (w-\wht)$, hence every part of the norms $\norm{X(0,t)}{u-\uht}$, $\norm{Y(0,t)}{w-\wht}$ is bounded as in the thesis. 
\end{proof}

According to the strategy proposed in \cite{art:amreinwihler}, it is now possible to perform a decomposition of the residual operators, by distinguishing the contribution from space discretization, time discretization and linearization as follows
:
\begin{subequations}
\begin{equation}
\begin{aligned}
&\langle R_1^h(t), \varphi \rangle = -\int_\Omega \frac{\uhkn - \uhnm}{\tau_n} \varphi \dx- \int_\Omega \EM\nabla \uhkn \cdot \nabla \varphi \dx\\
& \quad - \int_\Omega  \left[f(\uhkmn,\whkmn) + f_u(\uhkmn,\whkmn)(\uhkn - \uhkmn) \right. \\
& \quad \left. + f_w(\uhkmn,\whkmn)(\whkn - \whkmn) \right]\varphi\dx;
\end{aligned}
\label{eq:Rh1}
\end{equation}
\begin{equation}
\begin{aligned}
\langle R_1^\tau(t), \varphi \rangle &= - \int_\Omega \EM\nabla(\uht - \uhkn) \cdot \nabla \varphi \dx- \int_\Omega  \left[f(\uht,\wht) \right. \\
& \quad \left. - f(\uhkn,\whkn) \right]\varphi \dx;
\end{aligned}
\label{eq:Rt1}
\end{equation}
\begin{equation}
\begin{aligned}
\langle R_1^k(t), \varphi \rangle &= - \int_\Omega  \left[f(\uhkn,\whkn) - f_u(\uhkmn,\whkmn)(\uhkn - \uhkmn) \right. \\
	& \quad \left. - f_w(\uhkmn,\whkmn)(\whkn - \whkmn) - f(\uhkmn,\whkmn)\right]\varphi \dx;
\end{aligned}
\label{eq:Rk1}
\end{equation}
\end{subequations}
\begin{subequations}
\begin{equation}
\begin{aligned}
\langle R_2^h(t), \psi \rangle &= -\int_\Omega \frac{\whkn - \whnm}{\tau_n} \psi \dx - \int_\Omega \left[g(\uhkmn,\whkmn) \right. \\
& \quad + g_u(\uhkmn,\whkmn)(\uhkn - \uhkmn) \\
& \quad \left. + g_w(\uhkmn,\whkmn)(\whkn - \whkmn) \right]\psi \dx;
\end{aligned}
\label{eq:Rh2}
\end{equation}
\begin{equation}
\langle R_2^\tau(t), \psi \rangle = - \int_\Omega \left[g(\uht,\wht) - g(\uhkn,\whkn) \right]\psi \dx;
\label{eq:Rt2}
\end{equation}
\begin{equation}
\begin{aligned}
\langle R_2^k(t), \psi \rangle &= - \int_\Omega \left[g(\uhkn,\whkn) - g_u(\uhkmn,\whkmn)(\uhkn - \uhkmn) \right. \\
	& \quad \left. - g_w(\uhkmn,\whkmn)(\whkn - \whkmn) - g(\uhkmn,\whkmn)\right]\psi \dx.
\label{eq:Rk2}
\end{aligned}
\end{equation}
\end{subequations}
It is immediate to verify that $R_1(t) = R_1^h(t) + R_1^\tau(t) + R_1^k(t)$ in $H^*$ and $R_2(t)= R_2^h(t) + R_2^\tau(t) + R_2^k(t)$ in $L^2(\Omega)$; moreover, in view of the discrete problem \eqref{eq:Newt1}-\eqref{eq:Newt2}, the following orthogonality property holds:
\begin{equation}
\begin{aligned}
\langle R_1^h(t),\varphi_h \rangle  &= 0 \qquad &\forall \varphi_h \in V_h \\
\langle R_2^h(t),\psi_h \rangle &= 0 \qquad &\forall \psi_h \in V_h.
\end{aligned}
\label{eq:orthogonality}
\end{equation}

\section{\textit{A posteriori} estimators}
\label{estimators}
\textcolor{corr2}{We denote by $\widetilde{\mathcal{E}}_h^n$ the set of all faces of $\mathcal{T}_h^n$ and distinguish between the set of boundary faces $\widetilde{\mathcal{E}}_{h,\partial \Omega}^n$ and the set $\widetilde{\mathcal{E}}_{h,int}^n$ of the interior ones. Each face $E \in \widetilde{\mathcal{E}}_{h,int}^n$ is shared by two distinct elements, which we denote as $K_{E,1}$ and $K_{E,2}$; we define the jump of the conormal derivative across $E$ as
}
\[
\left[  k_E \normal_E \cdot \nabla u_h \right]_E = \left.\left(k_{K_{E,1}} \normal_{E,1} \cdot \nabla u_h|_{K_{E,1}} - k_{K_{E,2}} \normal_{E,2} \cdot \nabla u_h|_{K_{E,2}}\right)\right|_E,
\]
where $\normal_{E,1}$ and $\normal_{E,2}$ are outer the normals of $E$ with respect to $K_{E,1}$ and $K_{E,2}$, hence $\normal_{E,1} = -\normal_{E,2}$. For each face $E$ of $\widetilde{\mathcal{E}}_{h,\partial \Omega}^n$ (which belongs to a single element $K$ of the tessellation), we set
\[
\left[  k_E \normal_E \cdot \nabla u_h \right]_E = \left.\left(k_{K_{E}} \normal_E \cdot \nabla u_h|_{K_{E}}\right)\right|_E.
\]
We now introduce the following computable quantities which will appear in the \textit{a posteriori} estimates:
\par
\noindent \underline{\textit{Space indicators}}
\[
\begin{aligned}
\eta_k^n &= \left( \sum_{K \in \tilde{\Tau}_h^n} h_K^2 \norm{L^2(K)}{R_{K,1}}^2 + \sum_{E \in \tilde{\mathcal{E}}_h^n} h_E \norm{L^2(E)}{R_E}^2 + \norm{L^2(\Omega)}{\textcolor{corr2}{R_{\Omega,2}}}^2 \right)^\half \\
R_{K,1} &= \Bigl( -\frac{\uhkn - \uhnm}{\tau_n} + \nabla \cdot(M \nabla \uhkn) -  \left[ f(\uhkmn,\whkmn) \right. \\
& \quad \left. - f_u(\uhkmn,\whkmn)(\uhkn - \uhkmn) - f_w(\uhkmn,\whkmn)(\whkn - \whkmn) \right] \Bigr) \Bigr|_K \\
R_E &= \left[ k_E \normal_E \cdot \nabla u_h^n  \right]_E \\
R_{\Omega,2} &= -\frac{\whkn - \whnm}{\tau_n} - \left[g(\uhkmn,\whkmn) - g_u(\uhkmn,\whkmn)(\uhkn - \uhkmn) \right.  \\
& \quad - \left. g_w(\uhkmn,\whkmn)(\whkn - \whkmn)\right] .
\end{aligned}
\]
\underline{\textit{Time indicators}}
\[
\begin{aligned}
\vartheta_{k}^n &= \left(  \frac{1}{3} \norm{L^2(\Omega)}{M^{1/2} \nabla(\uhkn-\uhnm)}^2 + \frac{1}{\tau_n} \norm{L^2((t_{n-1},t_n)\times \Omega)}{P_{1}(t)}^2 + \frac{1}{\tau_n} \norm{L^2((t_{n-1},t_n)\times \Omega)}{P_{2}(t)}^2  \right)^\half \\
P_{1}(t) &= - \Bigl( f(\uht,\wht)-f(\uhkn,\whkn) \Bigr) \\
P_{2}(t) &= - \Bigl( g(\uht,\wht)- g(\uhkn,\whkn) \Bigr).
\end{aligned}
\]
\underline{\textit{Linearization indicators}}
\[
\begin{aligned}
\gamma_{k}^n &= \left(  \norm{L^2(\Omega)}{Q_{1}}^2 + \norm{L^2(\Omega)}{Q_{2}}^2 \right)^\half \\
Q_{1} &= -  \Bigl( f(\uhkn,\whkn) - f_u(\uhkmn,\whkmn)(\uhkn - \uhkmn)  \\
	& \quad - f_w(\uhkmn,\whkmn)(\whkn - \whkmn) - f(\uhkmn,\whkmn) \Bigr). \\
Q_{2} &= -\Bigl( g(\uhkn,\whkn) - g_u(\uhkmn,\whkmn)(\uhkn - \uhkmn)  \\
	& \quad - g_w(\uhkmn,\whkmn)(\whkn - \whkmn) - g(\uhkmn,\whkmn) \Bigr).
\end{aligned}
\]
The first main result of this section is the following \textit{a posteriori} upper bound:
\begin{theorem}
For each discrete solution $(\{u_{h,k}^n\},\{w_{h,k}^n\}$ with $n = 1,\ldots,N$, $k=1,\ldots,K_n$, collecting all $K_n$ in the multi-index $\mathbf{k} = [K_n]_{n=1}^N$ and definining $\uht,\wht$ as in \eqref{eq:interpolants}, it holds that for each $n = 1,\ldots,N$:
\begin{equation}
\begin{aligned}
 \left\{\norm{X(0,t_n)}{u-u_{h,\tau}^{\mathbf{k}}}^2 + \norm{Y(0,t_n)}{w-\wht}^2\right\}^\half &\lesssim \Bigl\{ \norm{L^2(\Omega)}{u_0 - \Pi_V^0 u_0}^2 + \norm{L^2(\Omega)}{w_0 - \Pi_h^0 w_0}^2 \\
& \quad + \sum_{m = 1}^n \tau_n((\eta_{k}^m)^2 + (\vartheta_{k,U}^m)^2 + (\gamma_{k,U}^m)^2 ) \Bigr\}^\half,
\end{aligned}
\label{eq:above}
\end{equation}
where the symbol $\lesssim$ denotes that an inequality holds up to a positive multiplicative constant independent of the space discretization parameter $h$.
\label{th:apost}
\end{theorem}

In order to prove Theorem \ref{th:apost}, we need a preliminary results dealing with the spatial residual operators only. 
\begin{lemma}
There exist two positive constants $c_\dag$, $c^{\dag}$ independent of $n$ s.t., for almost every $t \in (t_{n-1},t_n)$ and for each $n=1,\ldots,N$, it holds:
\begin{equation}
\frac{1}{c_\dag} \eta_k^n \leq \left(\norm{H^*}{R_1^h(t)}^2 + \norm{L^2(\Omega)}{R_2^h(t)}^2\right)^\half \leq c^\dag \eta_k^n.
\label{eq:hest}
\end{equation}
\label{lemma:hest}
\end{lemma}
\begin{proof}
\textcolor{corr2}{We follow the strategy outlined in \cite[Lemma 5.1]{art:verfurth} (see also \cite{art:amreinwihler}). In particular, since $R_1^k(t)$ and $R_2^k(t)$ are constant in time within each interval $(t_{n-1},t_n)$, estimates \eqref{eq:hest} can be proved by similar arguments as the ones employed for elliptic problems. We now consider $t \in (t_{n-1},t_n)$ and neglect the dependence of $R_1^h$, $R_2^h$ on $t$.}  Integrating by parts the expression of $R_1^h$
, we obtain that for each $\varphi \in H^1(\Omega)$
\[
\langle R_1^h, \varphi \rangle  =  \sum_{K \in \widetilde{\mathcal{T}}_h^n} \int_K R_{K,1} \varphi \dx + \sum_{E \in \widetilde{\mathcal{E}}_h^n} \int_E R_E \varphi \dx.
\]
We now introduce the Clément interpolation operator $I_h: H^1(\Omega) \rightarrow V_h^n$ (see \cite{art:clement}, \cite{book:brenner2007}); proceeding in a standard way (see, e.g., \cite{book:verfurthreview}) and employing the orthogonality properties in \eqref{eq:orthogonality} we have
\[
\begin{aligned}
|\langle R_1^h,\varphi \rangle| &= \left| \langle R_1^h,I_h \varphi \rangle + \langle R_1^h,\varphi - I_h \varphi \rangle \right| \leq \sum_{K \in \widetilde{\mathcal{T}}_h^n} \left|\int_K R_{K,1} (\varphi - I_h \varphi) \dx \right| + \sum_{E \in \widetilde{\mathcal{E}}_h^n} \left|\int_E R_E (\varphi - I_h \varphi) \dx\right|\\
& \leq c_1 \sum_{K \in \widetilde{\mathcal{T}}_h^n} h_K \norm{L^2(K)}{R_{K,1}} \norm{L^2(\widetilde{\omega_K})}{\nabla \varphi} + c_2 \sum_{E \in \widetilde{\mathcal{E}}_h^n} h_E^{\half}\norm{L^2(E)}{R_E} \norm{L^2(\widetilde{\omega_E})}{\nabla \varphi},
\end{aligned}
\]
where $\widetilde{\omega}_K$ (respectively, $\widetilde{\omega}_E$) is the union of all the elements of $\widetilde{\mathcal{T}}_h^n$ containing at least a vertex of $K$ (respectively, $E$). This entails that 
\[ 
\norm{H^*}{R_1^h} \leq \textcolor{corr2}{C^{\dag}} \left( \sum_{K \in \widetilde{\mathcal{T}}_h^n} h_K \norm{L^2(K)}{R_{K,1}} + \sum_{E \in \widetilde{\mathcal{E}}_h^n} h_E^{\half}\norm{L^2(E)}{R_E}\right).
\]
By an application of the Cauchy-Schwarz inequality it follows that $\norm{L^2(\Omega)}{R_2^h} \leq \norm{L^2(\Omega)}{R_{\Omega,2}} $, \textcolor{corr2}{hence the estimate from above in \eqref{eq:hest} holds with $c^\dag = \left(\max\{ 1,(C^\dag)^2 \}\right)^\half$}. \par
In order to prove the lower bound, we introduce
\[
W_n = \alpha  \sum_{K \in \widetilde{\mathcal{T}}_h^n} h_K^2 \phi_K R_{K,1} - \beta  \sum_{E \in \widetilde{\mathcal{E}}_h^n} h_E \phi_E R_E,
\]
with $\alpha, \beta > 0$, $\phi_K$, $\phi_E$ the baricentrical bubble functions respectively on $K$ and $\omega_E = K_{E,1} \cup K_{E,2}$. Analogously to \cite[Lemma 5.1]{art:verfurth}, we can show that 
\[
\langle R_1^h, W_n \rangle \geq  \left(\sum_{K \in \widetilde{\mathcal{T}}_h^n} h_K^2 \norm{L^2(K)}{R_K}^2 + \sum_{E \in \widetilde{\mathcal{E}}_h^n} h_E \norm{L^2(E)}{R_E}^2 \right)
\] 
and
\[
\norm{H^1(\Omega)}{W_n} \leq \textcolor{corr2}{C_\dag} \left(\sum_{K \in \widetilde{\mathcal{T}}_h^n} h_K^2 \norm{L^2(K)}{R_K}^2 + c_2 \sum_{E \in \widetilde{\mathcal{E}}_h^n} h_E \norm{L^2(E)}{R_E}^2 \right)^\half,
\]
\textcolor{corr2}{which entails that 
\begin{equation}
 \norm{H^*}{R_1^h} \geq \frac{1}{C_\dag} \left(\sum_{K \in \widetilde{\mathcal{T}}_h^n} h_K^2 \norm{L^2(K)}{R_K}^2 + c_2 \sum_{E \in \widetilde{\mathcal{E}}_h^n} h_E \norm{L^2(E)}{R_E}^2 \right)^\half.
\label{eq:Hstar}
\end{equation}
}
Regarding $R_2^h$, the following equality clearly holds
\[
\norm{L^2(\Omega)}{R_2^h}^2 = \int_\Omega R_2^h\ R_{\Omega,2} \dx = \norm{L^2(\Omega)}{R_{\Omega,2}}^2,
\]
\textcolor{corr2}{and this, together with \eqref{eq:Hstar} allows to conclude the lower bound in \eqref{eq:hest} with $\frac{1}{c_\dag} = \left(\min\left\{1,\frac{1}{C_\dag^2}\right\}\right)^\half$}.
\end{proof}
It is now possible to prove the upper bound \eqref{eq:above}.
\begin{proof}[Proof of Theorem \ref{th:apost}]
In view of \eqref{eq:B}, we only need to prove that, for each $n = 1,\ldots,N$, it holds
\begin{equation}
 \norm{L^2(t_{n-1},t_n,H^*)}{R_1}^2 + \norm{L^2((t_{n-1},t_n)\times\Omega)}{R_2}^2 \lesssim \tau_n \left((\eta_{k}^m)^2 + (\vartheta_{k}^m)^2 + (\gamma_{k}^m)^2 \right).
\label{eq:claim}
\end{equation}
According to Lemma \ref{lemma:hest}, 
\[
 \norm{H^*}{R_1^h(t)}^2 + \norm{L^2(\Omega)}{R_2^h(t)}^2 \lesssim (\eta_k^n)^2 \qquad \forall t \in (t_{n-1},t_n),
\]
and since by definition both $R_1^h$ and $R_2^h$ are constant in each interval $(t_{n-1},t_n)$, we conclude that
\begin{equation}
\norm{L^2(t_{n-1},t_n,H^*)}{R_1^h}^2 + \norm{L^2((t_{n-1},t_n)\times\Omega)}{R_2^h}^2 \lesssim \tau_n (\eta_k^n)^2.
\label{eq:4claim1}
\end{equation}
Moreover, it is immediate to verify via Cauchy-Schwarz inequality that
\[
 \norm{H^*}{R_1^k(t)}^2 + \norm{L^2(\Omega)}{R_2^k(t)}^2 \lesssim (\gamma_{k}^n)^2 \qquad \forall t \in (t_{n-1},t_n),
\]
which, integrating on $(t_{n-1},t_n)$ yields
\begin{equation}
\norm{L^2(t_{n-1},t_n,H^*)}{R_1^k}^2 + \norm{L^2((t_{n-1},t_n)\times\Omega)}{R_2^k}^2 \lesssim \tau_n (\gamma_{k}^n)^2.
\label{eq:4claim2}
\end{equation}
Eventually, again by the Cauchy-Schwarz inequality and employing \eqref{eq:interpolants}, for each $t \in (t_{n-1},t_n)$
\[
\begin{aligned}
\norm{H^*}{R_1^\tau(t)} + \norm{L^2(\Omega)}{R_2^\tau(t)} &\leq \mu_{max} \norm{L^2(\Omega)}{\nabla(\uht-\uhkn)} + \norm{L^2(\Omega)}{f(\uht,\wht)-f(\uhkn,\whkn)} \\
& \quad + \norm{L^2(\Omega)}{g(\uht,\wht)-g(\uhkn,\whkn)} \\
&\leq \frac{t_n-t}{\tau_n}\mu_{max}\norm{L^2(\Omega)}{\nabla(\uhkn-\uhnm)} + \norm{L^2(\Omega)}{f(\uht,\wht)-f(\uhkn,\whkn)} \\
& \quad + \norm{L^2(\Omega)}{g(\uht,\wht)-g(\uhkn,\whkn)}.
\end{aligned}
\]
Since $\int_{t_{n-1}}^{t_n} \left(\frac{t_n - t}{\tau_n}\right)^2 = \frac{\tau_n}{3}$, we get
\begin{equation}
\begin{aligned}
& \norm{L^2(t_{n-1},t_n,H^*)}{R_1^\tau}^2 + \norm{L^2((t_{n-1},t_n)\times\Omega)}{R_2^\tau}^2 \lesssim \frac{\tau_n}{3}\norm{L^2(\Omega)}{\nabla(\uhkn-\uhnm)}^2 \\
& \quad + \norm{L^2((t_{n-1},t_n)\times \Omega)}{f(\uht,\wht)-f(\uhkn,\whkn)}^2  \\
& \quad + \norm{L^2((t_{n-1},t_n)\times \Omega)}{g(\uht,\wht)-g(\uhkn,\whkn)}^2 \lesssim \tau_n (\theta_{k}^n)^2.
\end{aligned}
\label{eq:4claim3}
\end{equation}
By means of the triangular inequality, \eqref{eq:4claim1}, \eqref{eq:4claim2} and \eqref{eq:4claim3} we obtain \eqref{eq:claim}, and hence \eqref{eq:above}.
\end{proof}

\subsection{Efficiency of the estimators}
The upper estimate provided in \eqref{eq:above} holds for any choice of $\mathbf{k}$, i.e., the total number of Newton iterations $K_n$ performed in each interval $(t_{n-1},t_n)$ can be selected arbitrarily. We now prove a result of efficiency for our \textit{a posteriori} estimators, which holds true when a specific condition on the indices $K_n$ is satisfied. In particular, for each $n \geq 1$, we assume as in \cite[equation (3.12)]{art:vohralik} that there exists $K_n$ such that
\begin{equation}
\gamma_k^n \leq \sigma \eta_k^n,
\label{eq:sigma}
\end{equation}
being $\sigma < \frac{1}{c_\dag}$, where $c_\dag$ is the constant appearing in Lemma \ref{lemma:hest}. Such an hypothesis \textcolor{corr2}{can be understood as a stopping cryterion for the Newton algorithm associated to each timestep $n$. In particular, \eqref{eq:sigma} prescribes that an iteration $K_n$ is considered acceptable if the correspondent computable indicator of the linearization error is sufficiently smaller than the one associated to the space error.}\\
Moreover, we need to introduce the following assumption on the nonlinear terms $f$ and $g$: $\exists \lambda > 0$ (without loss of generality, we assume $\lambda \leq \mu_{min}$) such that, $\forall u_1,u_2,w_1,w_2 \in \R$,
\begin{equation}
\begin{aligned}
&\left(f(u_1,w_1)-f(u_2,w_2)\right)(u_1-u_2) + \left(g(u_1,w_1)-g(u_2,w_2)\right)(w_1-w_2) \\
&\quad \geq \lambda \left((u_1-u_2)^2 + (w_1-w_2)^2\right).
\end{aligned}
\label{eq:assumption}
\end{equation}
This assumption is verified under small modifications of the original problem by a large class of models, including Aliev-Panfilov, see Remark \ref{rem:assumption}. 

\begin{theorem}
Let $f,g$ satisfy \eqref{eq:assumption} and let $(\{u_{h,k}^n\}\{w_{h,k}^n\})$, $n = 0,\ldots,N$, $k = 0,\ldots,K_n$ be the fully discrete solution of \eqref{eq:formaforte} obtained by the Newton scheme \eqref{eq:Newt1}-\eqref{eq:Newt2}, satisfying assumption \eqref{eq:sigma} on the choice of $K_n$. Then,
\begin{equation}
\begin{aligned}
  \sqrt{\tau_n} ((\eta_{k}^n)^2 + (\vartheta_{k}^n)^2 + (\gamma_{k}^n)^2 )^\half \lesssim &  \left\{\norm{X(t_{n-1},t_n)}{u-\uht}^2 + \norm{Y(t_{n-1},t_n)}{w-\wht}^2\right\}^\half,
\end{aligned}
\label{eq:below}
\end{equation}
being $\uht$, $\wht$ the interpolants defined in \eqref{eq:interpolants}.
\label{th:efficiency}
\end{theorem}
\begin{proof}
First of all, we exploit the assumption \eqref{eq:assumption} on $f,g$ to obtain a useful inequality. Consider the temporal residual operators $R_1^\tau$, $R_2^\tau$ with test functions \textcolor{corr}{$\varphi_1 = \uht - \uhkn = -\frac{t_n-t}{\tau_n} (\uhkn - \uhnm)$, $\psi_1 = \wht - \whkn = -\frac{t_n-t}{\tau_n} (\whkn - \whnm)$}:
\[
\begin{aligned}
\langle R_1^\tau,\varphi_1 \rangle + \langle R_2^\tau,\psi_1 \rangle &\geq \mu_{min} \norm{L^2(\Omega)}{\nabla(\uht - \uhkn)}^2 + \lambda\left( \norm{L^2(\Omega)}{\uht - \uhkn}^2 +  \norm{L^2(\Omega)}{\wht - \whkn}^2 \right) \\
&\geq \lambda \textcolor{corr}{\left(\frac{t_{n}-t}{\tau_n}\right)^2} \left( \norm{H^1(\Omega)}{\uhkn- \uhnm}^2 +  \norm{L^2(\Omega)}{\whkn- \whnm}^2 \right).
\end{aligned}
\]
\textcolor{corr}{We recall that
\[
 \langle R_1^\tau,\varphi_1 \rangle + \langle R_2^\tau,\psi_1 \rangle = \langle R_1,\varphi_1 \rangle + \langle R_2,\psi_1 \rangle - \langle R_1^h,\varphi_1 \rangle - \langle R_2^h,\psi_1 \rangle - \langle R_1^k,\varphi_1 \rangle - \langle R_2^k,\psi_1 \rangle;
\]
when integrating in time, we can bound the right-hand side by considering two terms at a time as follows:
\[
\begin{aligned}
&\int_{t_{n-1}}^{t_n} \left|\langle R_1,\varphi_1 \rangle + \langle R_2,\psi_1 \rangle \right| dt \leq \int_{t_{n-1}}^{t_n} \left(\norm{H^*}{R_1}\norm{H^1(\Omega)}{\varphi_1} + \norm{L^2(\Omega)}{R_2}\norm{L^2(\Omega)}{\psi_1}\right) dt \\
& \quad \leq \int_{t_{n-1}}^{t_n} \left(\norm{H^*}{R_1} + \norm{L^2(\Omega)}{R_2}\right) \left(\norm{H^1(\Omega)}{\varphi_1} + \norm{L^2(\Omega)}{\psi_1}\right) dt\\
& \quad \leq \int_{t_{n-1}}^{t_n} \frac{t_n-t}{\tau_n} \left(\norm{H^*}{R_1} + \norm{L^2(\Omega)}{R_2}\right) \left(\norm{H^1(\Omega)}{\uhkn-\uhnm} + \norm{L^2(\Omega)}{\whkn-\whnm}\right) dt \\
& \quad \leq 2 \left(\norm{H^1(\Omega)}{\uhkn-\uhnm}^2 + \norm{L^2(\Omega)}{\whkn-\whnm}^2\right)^{1/2} \\
& \qquad \int_{t_{n-1}}^{t_n} \frac{t_n-t}{\tau_n} \left(\norm{H^*}{R_1}^2 + \norm{L^2(\Omega)}{R_2}^2\right)^{1/2} dt \\
& \quad \leq 2 \left(\norm{H^1(\Omega)}{\uhkn-\uhnm}^2 + \norm{L^2(\Omega)}{\whkn-\whnm}^2\right)^{1/2}  \\
& \qquad \left( \int_{t_{n-1}}^{t_n} \frac{(t_n-t)^2}{\tau_n^2} dt \right)^{1/2} \left( \norm{L^2(t_{n-1},t_n;H^*)}{R_1}^2 + \norm{L^2(t_{n-1},t_n;L^2(\Omega))}{R_2}^2 \right)^{1/2} \\
& \quad \leq 2 \frac{\sqrt{\tau_n}}{\sqrt{3}} c^* \norm{XY}{err} \left( \norm{H^1(\Omega)}{\uhkn-\uhnm}^2 + \norm{L^2(\Omega)}{\whkn-\whnm}^2 \right)^{1/2},
\end{aligned}
\]
where we set $\norm{XY}{err} \defeq \left(\norm{X(t_{n-1},t_n)}{u-\uht}^2+\norm{Y(t_{n-1},t_n)}{w-\wht}^2\right)^\half$ and we made use of \eqref{eq:A} and of the Jensen inequality $A+B \leq \sqrt{2}(A^2 + B^2)^\half$. Moreover, via \eqref{eq:hest} we get
\[
\begin{aligned}
&\int_{t_{n-1}}^{t_n} \left|\langle R_1^h,\varphi_1 \rangle + \langle R_2^h,\psi_1 \rangle\right|dt \leq \int_{t_{n-1}}^{t_n} \left(\norm{H^*}{R_1^h}\norm{H^1(\Omega)}{\varphi_1} + \norm{L^2(\Omega)}{R_2^h}\norm{L^2(\Omega)}{\psi_1} \right) dt \\
& \quad \leq \int_{t_{n-1}}^{t_n} \frac{t_n-t}{\tau_n} \left(\norm{H^*}{R_1^h}+\norm{L^2(\Omega)}{R_2^h}\right)\left(\norm{H^1(\Omega)}{\uhkn-\uhnm} + \norm{L^2(\Omega)}{\whkn-\whnm}\right) dt \\
& \quad \leq 2\int_{t_{n-1}}^{t_n} \frac{t_n-t}{\tau_n} dt \left(\norm{H^*}{R_1^h}^2+\norm{L^2(\Omega)}{R_2^h}^2\right)^{1/2} \left(\norm{H^1(\Omega)}{\uhkn-\uhnm}^2 + \norm{L^2(\Omega)}{\whkn-\whnm}^2\right)^{1/2} \\
& \quad \leq \tau_n c^\dag \eta_k^n \left(\norm{H^1(\Omega)}{\uhkn-\uhnm}^2 + \norm{L^2(\Omega)}{\whkn-\whnm}^2\right)^{1/2}.
\end{aligned}
\]
Eventually, by the definition of $R_1^k, R_2^k$ and $\gamma_k^n$,
\[
\begin{aligned}
&\int_{t_{n-1}}^{t_n} \left| \langle R_1^k,\varphi_1 \rangle + \langle R_2^k,\psi_1 \rangle \right| dt \leq \int_{t_{n-1}}^{t_n} \left| \int_\Omega Q_1\varphi_1 \dx + \int_\Omega Q_2 \psi_1 \dx \right| dt \\
& \quad \leq \int_{t_{n-1}}^{t_n} \left(\norm{L^2(\Omega)}{Q_1}\norm{L^2(\Omega)}{\varphi_1}+\norm{L^2(\Omega)}{Q_2}\norm{L^2(\Omega)}{\psi_1}\right) dt \\
& \quad \leq \int_{t_{n-1}}^{t_n}\frac{t_n-t}{\tau_n} dt \left( \norm{L^2(\Omega)}{Q_1}^2+\norm{L^2(\Omega)}{Q_2}^2\right)^{1/2}\left(\norm{H^1(\Omega)}{\uhkn-\uhnm}^2 + \norm{L^2(\Omega)}{\whkn-\whnm}^2\right)^{1/2} \\
& \quad = \tau_n \gamma_k^n \left(\norm{H^1(\Omega)}{\uhkn-\uhnm}^2 + \norm{L^2(\Omega)}{\whkn-\whnm}^2\right)^{1/2}.
\end{aligned}
\]
This allows to conclude that
}
\begin{equation}
\lambda \frac{\tau_n}{3} \left(\norm{H^1(\Omega)}{\uhkn- \uhnm}^2 +  \norm{L^2(\Omega)}{\whkn- \whnm}^2 \right)^\half \leq \textcolor{corr}{\frac{2\sqrt{\tau_n}}{\sqrt{3}}} c^* \norm{XY}{err} + \textcolor{corr}{\tau_n} c^\dag \eta_k^n + \textcolor{corr}{\tau_n} \gamma_{k}^n.
\label{eq:0}
\end{equation}
\par
We focus now on the spatial estimator $\eta_k^n$. According to the proof of Lemma \ref{lemma:hest}, \textcolor{corr2}{for the particular choice of test functions $\varphi_2 = W_n$, $\psi_2 = R_{\Omega,2}$, it holds that}
\[
\langle R_1^h,\varphi_2 \rangle + \langle R_2^h,\psi_2 \rangle \geq (\eta_k^n)^2, \qquad \left(\norm{H^1(\Omega)}{\varphi_2}^2+\norm{L^2(\Omega)}{\psi_2}^2\right)^\half\leq c_\dag \eta_k^n
\] 
whence
\[
\frac{1}{c_\dag} \eta_k^n \left(\norm{H^1(\Omega)}{\varphi_2}^2+\norm{L^2(\Omega)}{\psi_2}^2\right)^\half \leq \langle R_1^h,\varphi_2 \rangle + \langle R_2^h,\psi_2 \rangle.
\]
By the decomposition of the residual, $R_1^h = R_1 -R_1^\tau - R_1^k$ and $R_2^h = R_2 -R_2^\tau - R_2^k$. Moreover,
\[
\begin{aligned}
\left|\langle R_1^\tau,\varphi_2 \rangle + \langle R_2^\tau,\psi_2 \rangle\right| &\leq \mu_{max} \int_\Omega \left| \nabla(\uht - \uhkn) \cdot \nabla \varphi_2 \right|\dx + \int_\Omega \left| \left[f(\uht,\wht) - f(\uhkn,\whkn) \right]\varphi_2 \right|\dx \\
& \quad + \int_\Omega \left| \left[g(\uht,\wht) - g(\uhkn,\whkn) \right]\psi_2 \right|\dx\\
& \leq \mu_{max}\norm{L^2(\Omega)}{\nabla(\uht - \uhkn)}\norm{L^2(\Omega)}{\nabla \varphi_2} \\
& \quad + K_f \left(\norm{L^2(\Omega)}{\uht - \uhkn}+\norm{L^2(\Omega)}{\wht - \whkn}\right)\norm{L^2(\Omega)}{\varphi_2} \\
& \quad + K_g \left(\norm{L^2(\Omega)}{\uht - \uhkn}+\norm{L^2(\Omega)}{\wht - \whkn}\right)\norm{L^2(\Omega)}{\psi_2}\\
& \leq K_{fg} \left(\norm{H^1(\Omega)}{\uht - \uhkn}+\norm{L^2(\Omega)}{\wht - \whkn}\right) \left(\norm{H^1(\Omega)}{\varphi_2} + \norm{L^2(\Omega)}{\psi_2}\right),
\end{aligned}
\]
where $K_f$ and $K_g$ are the Lipschitz constants of $f$ and $g$ and $K_{fg} = max\{\mu_{max},K_f,K_g\}$. Exploiting the Cauchy-Schwarz and the Jensen inequalities and the definition of $\gamma_k^n$,
\[
\begin{aligned}
\frac{1}{c_\dag} \eta_k^n \leq& 2 \left(\norm{H^*}{R_1}^2 + \norm{L^2(\Omega)}{R_2}^2\right)^\half + 2\gamma_{k}^n + 2 K_{fg}\left(\norm{H^1(\Omega)}{\uht - \uhkn}^2+\norm{L^2(\Omega)}{\wht - \whkn}^2\right)^\half ,
\end{aligned}
\]
and since $\uht - \uhkn = \frac{t_n - t}{\tau_n}(\uhkn-\uhnm)$, we have
\begin{equation}
\begin{aligned}
\frac{1}{c_\dag} \eta_k^n \leq& 2 \left(\norm{H^*}{R_1}^2 + \norm{L^2(\Omega)}{R_2}^2\right)^\half + 2\gamma_{k}^n \\
& + 2 \frac{t_n - t}{\tau_n}K_{fg} \left(\norm{H^1(\Omega)}{\uhkn-\uhnm}^2+\norm{L^2(\Omega)}{\whkn-\whnm}^2\right)^\half.
\end{aligned}
\label{eq:starstar}
\end{equation}
Now, we take advantage of the strategy used in the proof of the lower bound in \cite{art:verfurth}, in particular, choosing a positive $\alpha$, we multiply the inequality \eqref{eq:starstar} by $(\alpha+1)\left(\frac{t-t_{n-1}}{\tau_n}\right)^\alpha$ and integrate from $t_{n-1}$ to $t_n$. We observe that
\[
\begin{aligned}
 &\int_{t_{n-1}}^{t_n} (\alpha+1)\left(\frac{t-t_{n-1}}{\tau_n}\right)^\alpha dt= \tau_n ; \\
 &\int_{t_{n-1}}^{t_n} \left(\frac{t-t_{n-1}}{\tau_n}\right)^\alpha (\alpha+1)\left(\frac{t_{n}-t}{\tau_n}\right)dt = \tau_n \frac{1}{\alpha+2}; \\
 &\int_{t_{n-1}}^{t_n} (\alpha+1)\left(\frac{t-t_{n-1}}{\tau_n}\right)^\alpha \left(\norm{H^*}{R_1}^2 + \norm{L^2(\Omega)}{R_2}^2\right)^\half dt\leq \\
& \quad\sqrt{\tau_n} \frac{\alpha +1}{\sqrt{2\alpha+1}} \left( \norm{L^2(t_{n-1},t_n;H^*)}{R_1}^2 + \norm{L^2((t_{n-1},t_n)\times \Omega)}{R_2}^2 \right)^\half .
\end{aligned}
\]
Thus, we obtain (applying \eqref{eq:A} and \eqref{eq:0})
\[
\begin{aligned}
\frac{1}{c_\dag} \tau_n \eta_k^n \leq& \sqrt{\tau_n} \frac{\alpha +1}{\sqrt{2\alpha+1}} c^*\norm{XY}{err} + \tau_n \gamma_{k}^n \\
& \quad + \tau_n \frac{2}{\alpha+2} K_{fg} \left( \norm{L^2(\Omega)}{\uhkn- \uhnm}^2 + \norm{L^2(\Omega)}{\whkn- \whnm}^2 \right)^{1/2} \\
\leq& \sqrt{\tau_n} \frac{\alpha +1}{\sqrt{2\alpha+1}} c^*\norm{XY}{err} + \tau_n \gamma_{k}^n \\
& \quad + \frac{\textcolor{corr}{6} K_{fg}}{(\alpha+2)\lambda} \left(\textcolor{corr}{\frac{2\sqrt{\tau_n}}{\sqrt{3}}} c^*  \norm{XY}{err} + \tau_n c^\dag \eta_k^n + \tau_n \gamma_{k}^n \right).
\end{aligned}
\]
Taking advantage of the assumtpion \eqref{eq:sigma} and dividing by $\sqrt{\tau_n}$, we get
\begin{equation}
\frac{1}{c_\dag} \sqrt{\tau_n} \eta_k^n \leq  c^*\left( \frac{\alpha +1}{\sqrt{2\alpha+1}} + \frac{\textcolor{corr}{4\sqrt{3}} K_{fg}}{\lambda (\alpha+2)} \right) \norm{XY}{err} + \sqrt{\tau_n} \left( \frac{\textcolor{corr}{6} K_{fg}(\sigma + c^\dag)}{\lambda (\alpha+2)} + \sigma \right)\eta_k^n.
\label{eq:almost}
\end{equation}
\textcolor{corr2}{Since by assumption \eqref{eq:sigma} $\frac{1}{c_\dag} - \sigma >0$, selecting
\[
\alpha = \max\left\{0, \frac{6 K_{fg}(c^\dag + \sigma)c_\dag}{\lambda (1 - {c_\dag}\sigma)} -2 \right\}
\]
we can ensure that
\[
\frac{6 K_{fg}(\sigma + c^\dag)}{\lambda (\alpha+2)} + \sigma < \frac{1}{c_\dag}.
\]
}
Thus, we deduce
\begin{equation}
	\sqrt{\tau_n} \eta_k^n \lesssim  \norm{XY}{err};
\label{eq:3}
\end{equation}
from now on, we omit the explicit expression of the constants in front of each term in the inequality. 
As an immediate consequence, again by \eqref{eq:sigma}, we infer
\begin{equation}
	\sqrt{\tau_n}\gamma_k^n \leq \sqrt{\tau_n} \sigma \eta_k^n \lesssim \norm{XY}{err}.
\label{eq:42}
\end{equation}
We now focus on $\theta_{k}^n$. By definition,
\[
\begin{aligned}
(\vartheta_{k}^n)^2 = & \frac{1}{3} \norm{L^2(\Omega)}{M^{1/2}\nabla(\uhkn-\uhnm)}^2 + \frac{1}{\tau_n} \norm{L^2((t_{n-1},t_n)\times \Omega)}{f(\uht,\wht)-f(\uhkn,\whkn)}^2 \\&+ \frac{1}{\tau_n} \norm{L^2((t_{n-1},t_n)\times \Omega)}{g(\uht,\wht)-g(\uhkn,\whkn)}^2 
 \\
& \leq \frac{1}{3} \norm{L^2(\Omega)}{M^{1/2}\nabla(\uhkn-\uhnm)}^2 + \frac{K_{fg}^2}{\tau_n} \int_{t_{n-1}}^{t_n} \left( \norm{L^2(\Omega)}{\uht-\uhkn}^2 + \norm{L^2(\Omega)}{\wht-\whkn}^2 \right)dt\\
& \leq \frac{1}{3} \norm{L^2(\Omega)}{M^{1/2}\nabla(\uhkn-\uhnm)}^2 + \frac{K_{fg}^2}{3} \left( \norm{L^2(\Omega)}{\uhkn-\uhnm}^2 + \norm{L^2(\Omega)}{\whkn-\whnm}^2 \right)\\
& \leq \frac{K_{fg}^2}{3} \left( \norm{H^1(\Omega)}{\uhkn-\uhnm}^2 + \norm{L^2(\Omega)}{\whkn-\whnm}^2 \right).
\end{aligned}
\]
Therefore, in view of \eqref{eq:0}
\[
\begin{aligned}
\vartheta_{k}^n \leq & \textcolor{corr}{\frac{K_{fg}}{\sqrt{3}}}\left( \norm{H^1(\Omega)}{\uhkn-\uhnm}^2 + \norm{L^2(\Omega)}{\whkn-\whnm}^2 \right)^\half \\ \leq & \textcolor{corr}{\frac{\sqrt{3}K_{fg}}{\lambda \tau_n}} \left( \textcolor{corr}{\frac{2 \sqrt{\tau_n}}{\sqrt{3}}} c^*  \norm{XY}{err} + \textcolor{corr}{\tau_n} c^\dag \eta_k^n + \tau_n \gamma_{k}^n \right),
\end{aligned}
\]
and eventually (using \eqref{eq:3} and \eqref{eq:42})
\begin{equation}
\sqrt{\tau_n} \vartheta_{k}^n \leq \frac{K_{fg}^2}{\lambda} \left( \textcolor{corr}{\frac{2 \sqrt{\tau_n}}{\sqrt{3}}} c^* \norm{XY}{err} + \textcolor{corr}{\sqrt{\tau_n}} c^\dag \eta_k^n + \sqrt{\tau_n} \gamma_{k}^n \right) \lesssim \norm{XY}{err}.
\label{eq:52}
\end{equation}
Eventually, collecting the results \eqref{eq:3}, \eqref{eq:42}, \eqref{eq:52} we conclude that
\begin{equation}
 \sqrt{\tau_n}\left( (\eta_k^n)^2 + (\theta_k^n)^2 + (\gamma_k^n)^2 \right)^\half \leq \sqrt{\tau_n}\left( \eta_k^n + \theta_k^n + \gamma_k^n \right) \lesssim \norm{XY}{err}.
\label{eq:star}
\end{equation}
\end{proof}


\begin{remark}
Assumption \eqref{eq:assumption} is in general not satisfied by $f$ and $g$ as in \eqref{eq:AP}. In particular, inequality \eqref{eq:assumption} holds with a possibly negative constant, $-\tilde{K}$. This can be deduced by mean value theorem, exploiting the fact that $f,g$ in \eqref{eq:AP} are continuously differentiable and take values on a bounded subset of $\R^2$ due to the uniform \textit{a priori} bounds on the solutions prescribed in Proposition \ref{prop:prel}. However, we can introduce a change of variable in the original problem \eqref{eq:formaforte}: for a positive $\lambda$, we set $\tilde{u} = e^{-(\tilde{K}+\lambda)t}u$ and $\tilde{w} = e^{-(\tilde{K}+\lambda)t}w$. It holds $\partial_t \tilde{u} = -(\tilde{K}+\lambda) \tilde{u} + e^{-(\tilde{K}+\lambda)t} \partial_t u$, and $(\tilde{u},\tilde{w})$ is the solution of 
\[
\left\{
\begin{aligned}
\partial_t \tilde{u} - \nabla \cdot(\EM\nabla \tilde{u}) + \tilde{f}(\tilde{u},\tilde{w}) &= 0 \qquad &\text{in } \Omega \times (0,T) \\
\partial_t \tilde{w} + \tilde{g}(\tilde{u},\tilde{w}) &= 0 \qquad &\text{in }  \Omega \times (0,T), \\
\end{aligned}
\right.
\]
where $\tilde{f} = e^{-(\tilde{K}+\lambda)t} f(e^{(\tilde{K}+\lambda)t}\tilde{u},e^{(\tilde{K}+\lambda)t}\tilde{w}) + (\tilde{K}+\lambda)\tilde{u}$ and $\tilde{g}$ (analogously defined) satisfy \eqref{eq:assumption}.
\label{rem:assumption}
\end{remark}

\begin{remark}
In the particular case where the source of error coming from the linearization process is disregarded, the simplified counterpart of Theorem \ref{th:apost} holds with the only estimators $\eta^n$, $\theta^n$ defined as
%
\begin{equation}
\begin{aligned}
\eta^n =& \Biggl( \sum_{K \in \tilde{\Tau}_h^n} h_K^2 \norm{L^2(K)}{\frac{u_h^n-u_h^{n-1}}{\tau_n} + \nabla \cdot(M \nabla u_h^n) + f(u_h^n,w_h^n)}^2 + \sum_{E \in \tilde{\mathcal{E}}_h^n}h_K \norm{L^2(E)}{[\nabla u_h^n \cdot n_E]}^2 \\
& + \sum_{K \in \tilde{\Tau}_h^n} \norm{L^2(K)}{\frac{w_h^n-w_h^{n-1}}{\tau_n} + g(u_h^n,w_h^n)}^2 \Biggr)^\half \\
\vartheta^n =& \Biggl(  \frac{1}{3} \norm{L^2(\Omega)}{M^{1/2} \nabla(u_h^n - u_h^{n-1})}^2 + \frac{1}{\tau_n} \norm{L^2((t_{n-1},t_n)\times \Omega)}{(f(\uhti)-f(u_h^n))}^2 \\
& + \frac{1}{\tau_n} \norm{L^2((t_{n-1},t_n)\times \Omega)}{(g(\uhti)-g(u_h^n))}^2  \Biggr)^\half,
\end{aligned}
\label{eq:estEI}
\end{equation}
being $\uhti = \frac{t_n - t}{\tau_n} u_h^{n-1} + \frac{t-t_{n-1}}{\tau_n} u_h^n$ and $\whti = \frac{t_n - t}{\tau_n} w_h^{n-1} + \frac{t-t_{n-1}}{\tau_n} w_h^n$. An efficiency result analogous to Theorem \ref{th:efficiency} holds with the same estimators, clearly without requiring \eqref{eq:sigma}.
\label{rem:EI}
\end{remark}

\section{Numerical experiments}
\label{results}
We now numerically assess the validity of the derived \textit{a posteriori} estimates. We consider the following two-dimensional setup: the domain $\Omega$ is the square $(0,1)^2$, whereas the time interval is set equal to $(0,16)$. All the experiments are performed in an isotropic tissue, whence $M$ is a scalar coefficient. We consider the initial data
\[
u_0 = e^{-\frac{(x-1)^2+y^2}{0.25}}, \qquad w_0 = 0,
\]
whereas the value of the constants of the problem are reported in Table 1.
\begin{table}[h!]
\centering
\begin{tabular}[t]{cccc}
$M$ & $A$ & $\epsilon$ & $a$ \\
\hline
$1$ & $8$ & $0.2$ & $0.15$\\
\end{tabular}\caption{Values of the main parameters of the model}
\label{tab11:param}
\end{table}
We report in Figure \ref{fig:U} several snapshots of the evolution of the electrical potential $u$ throughout time. The results are obtained via the Newton-Galerkin scheme in \eqref{eq:Newt1}-\eqref{eq:Newt2}, making use of the same computational mesh $\mathcal{T}_h$ for each instant, with maximum diameter $h = 0.0125$ and a fixed timestep $\tau = 0.025$. As an exit criterion for the Newton iterations we check if the distance between two following iterations (measured in $H^1$ and $L^2$ norm respectively for $u$ and $w$) is below a suitable tolerance, which we set as $tol = 10^{-14}$. In accordance with experimental observations (see, e.g., \cite{book:pavarino}), the nonlinear dynamics shows a first quick propagation of the stimulus in the tissue and, after a plateau phase, a slow decrease of the electrical potential.
\begin{figure}[h!]
	\centering
	\subfloat[$t_1$ = 0]{
	 	\includegraphics[width=0.31\textwidth]{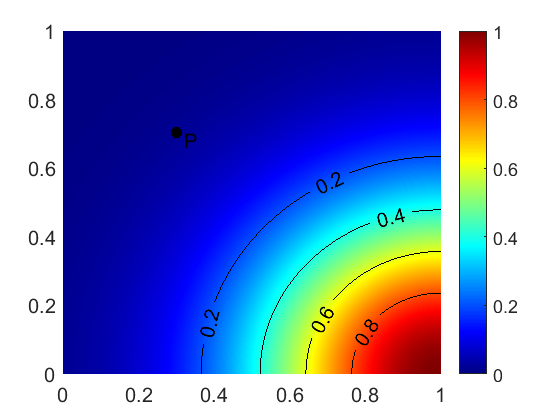}
	}
	\subfloat[$t_2$ = 0.5]{
	 	\includegraphics[width=0.31\textwidth]{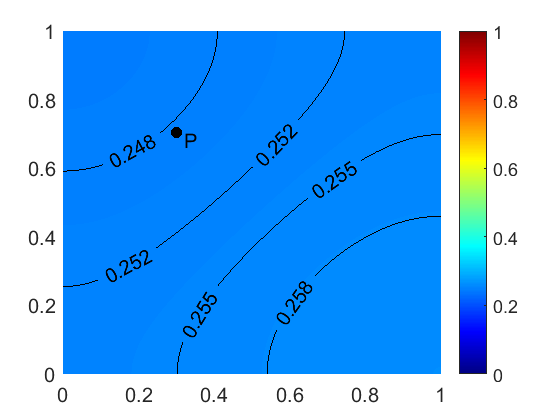}
	}
	\subfloat[$t_3$ = 2.5]{
	 	\includegraphics[width=0.31\textwidth]{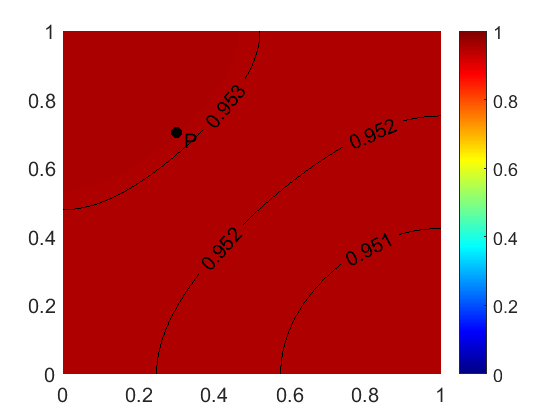}
	}\\
	\subfloat[$t_4$ = 7]{
	 	\includegraphics[width=0.31\textwidth]{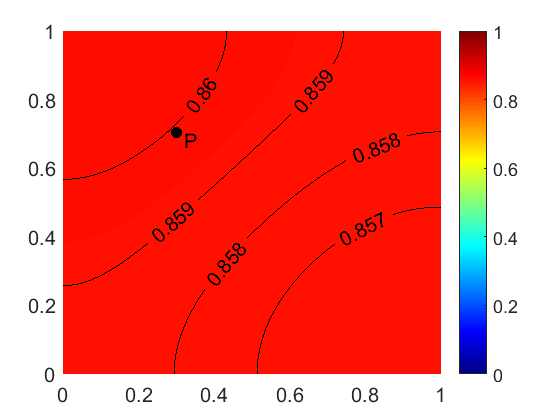}
	}
	\subfloat[$t_5$ = 10]{
	 	\includegraphics[width=0.31\textwidth]{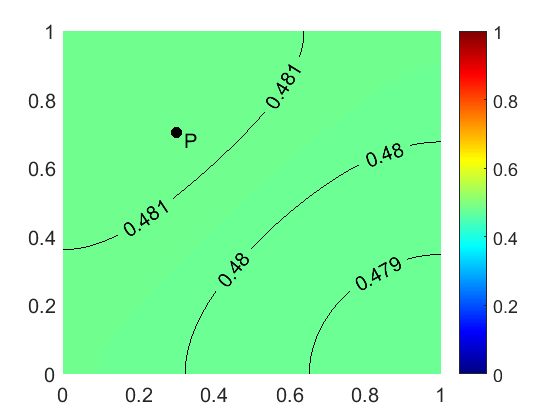}
	}
	\subfloat[$t_6$ = 14]{
	 	\includegraphics[width=0.31\textwidth]{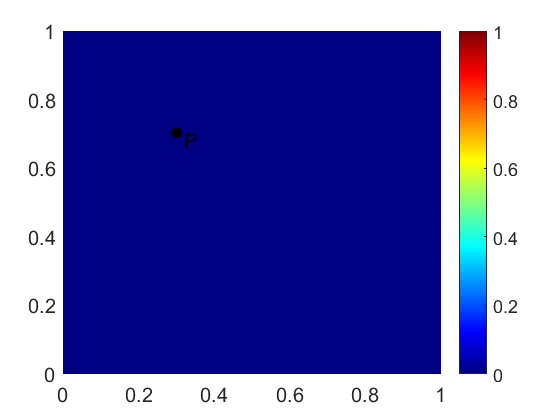}
	} \\
	\subfloat[Evolution at a specific point $P$]{
	 	\includegraphics[width=0.5\textwidth]{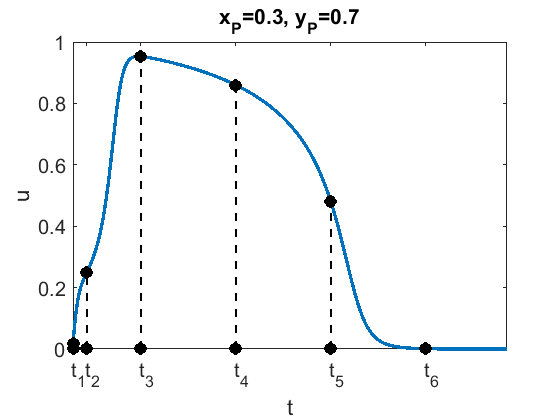}
	}
\caption{Snapshots of the evolution of the electrical potential. In Figures (a)-(f) the contour plots are shown in some selected instants $t_1, \ldots, t_6$. Figure (g) reports the value of the electrical potential in a specific point $P$; the instants $t_1, \ldots, t_6$ are remarked.}
\label{fig:U}
\end{figure}

\subsection{Spatial and temporal analysis}
We now verify the validity of the estimates stated in Theorem \ref{th:apost}. Due to the lack of an analytical expression for the solution of \eqref{eq:formaforte}, we need to build a high-fidelity numerical solution $(\tilde u,\tilde w)$. In particular, we employ a reference fine mesh with $h_{ref} = 4\cdot 10^{-3}$ and a time step $\tau_{ref} = 2 \cdot 10^{-3}$ to solve the  Newton scheme \eqref{eq:Newt1}-\eqref{eq:Newt2}, where $tol = 10^{-15}$ is employed to make negligible the linearization error (see Remark \ref{rem:EI}). Employing $(\tilde u,\tilde w)$ it is possible to compute the error associated to different discrete solutions, obtained with different values of $h$ and $\tau$, and to assess the validity of the \textit{a posteriori} error estimates introduced in Theorem \ref{th:apost} employing in particular the estimators defined in \eqref{eq:estEI}. 
\par
In Figure \ref{fig:above} we report the numerical verification of the upper bound \eqref{eq:above} for two different choices of the discretization parameters $h$ and $\tau$. Each line is piecewise constant on every interval $(t_{n-1},t_n)$. The red line represents the norm of the error on the interval $(0,t_n)$ (see the left-hand side of \eqref{eq:above} for its precise definition) computed with respect to the high-fidelity solution, whereas the blue line shows the sum of the estimators in each interval until $t_n$ (see the left-hand side of \eqref{eq:above}). In this case the upper bound holds with constant $1$.
\begin{figure}[h!]
\centering
	\subfloat[$h = 0.05$, $\tau = 0.1$]{
	 	\includegraphics[width=0.45\textwidth]{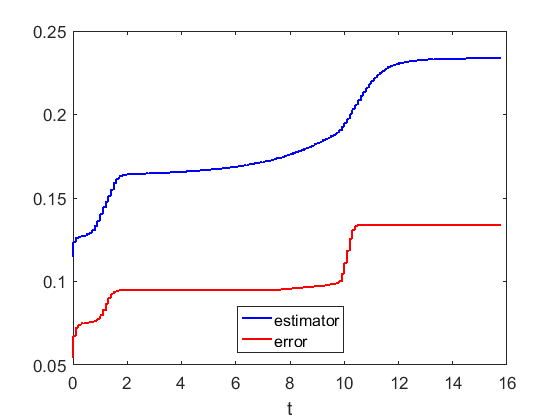}
	}
	\subfloat[$h = 0.0125$, $\tau = 0.025$]{
	 	\includegraphics[width=0.45\textwidth]{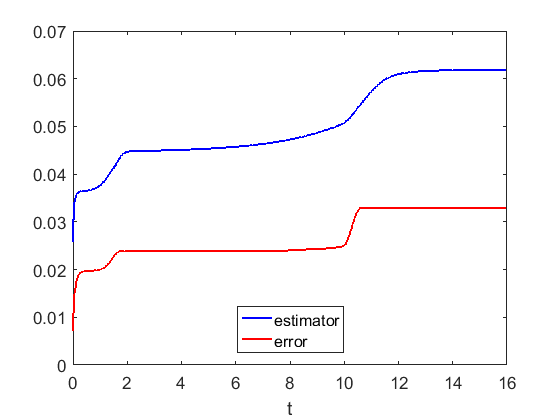}
	}
\caption{Assessment of the upper bound}%
\label{fig:above}%
\end{figure}
\par
Moreover, in Figure \ref{fig:diagonal} we investigate the convergence rates for both the \textit{a posteriori} estimator and the error norm with respect to the mesh size $h$ and the timestep $\tau$.
The results are obtained by linearly reducing both $h$ and $\tau$ at the same time. The convergence history reported in Figure \ref{fig:diagonal} shows that the error decays with linear rate, as expected from the \textit{a priori} estimate in Theorem \ref{th:apriori2}, and the \textit{a posteriori} estimator decays with the same (linear) rate.
\par
\begin{figure}%
\begin{minipage}[t]{.5\textwidth}
\vspace{0pt}
\raggedleft
	\includegraphics[width=\textwidth]{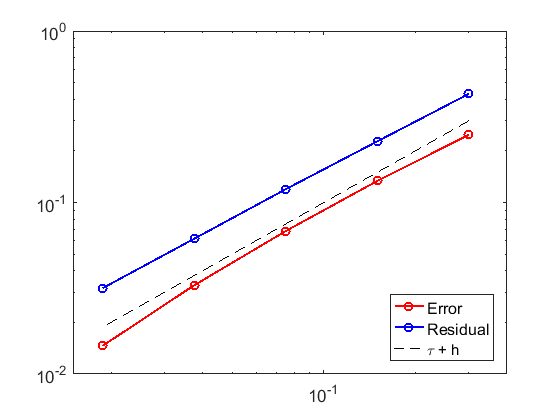}
\vspace{0pt}
\end{minipage}
\begin{minipage}[t]{.5\textwidth}
\vspace{1cm}
\hspace{1cm}
\raggedright
\begin{tabular}{ccc}
	  $h$ & $\tau$ & $tol$ \\
		\hline 
		$0.1$ & $0.2$ & $10^{-14}$ \\
		$0.05$ & $0.1$ & $10^{-14}$ \\
		$0.025$ & $0.05$ & $10^{-14}$ \\
		$0.0125$ & $0.025$ & $10^{-14}$ \\
		$0.00625$ & $0.0125$ & $10^{-14}$
	\end{tabular}
	\vspace{0pt}
\end{minipage}
\caption{Convergence analysis in $h$ and $\tau$}%
\label{fig:diagonal}%
\end{figure}

\subsection{Linearization analysis}
We now numerically assess the validity of the \textit{a posteriori} estimate concerning the linearization error. In order to reduce as much as possible the numerical error induced by spatial and temporal approximations, we perform the the numerical experiments with the same discretization parameters ($h_{ref} = 4\cdot 10^{-3}$, $\tau_{ref} = 2 \cdot 10^{-3}$) employed to build the high-fidelity numerical solution. Selecting an instant $t_n$, we compute several iterations of the Newton scheme \eqref{eq:Newt1}-\eqref{eq:Newt2} until the convergence criterion 
is satisfied with $tol = 10^{-15}$. The iterative scheme produces a sequence $\{\uhkn, \whkn\}_{k = 0,\ldots,K}$. Then, for each $k$ we compute $\gamma_k^n$ and compare it with the linearization error. In Figure \ref{fig:newton} we report the described comparison at $t_n = 2.5$ and $t_n = 10$.
\begin{figure}[h!]
\centering
	\subfloat[$t_n = 2.5$; accepted at iteration $5$]{
	 	\includegraphics[width=0.45\textwidth]{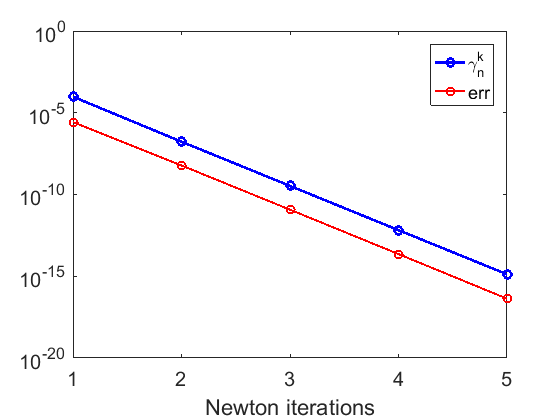}
	}
	\subfloat[$t_n = 10$; accepted at iteration $4$]{
	 	\includegraphics[width=0.45\textwidth]{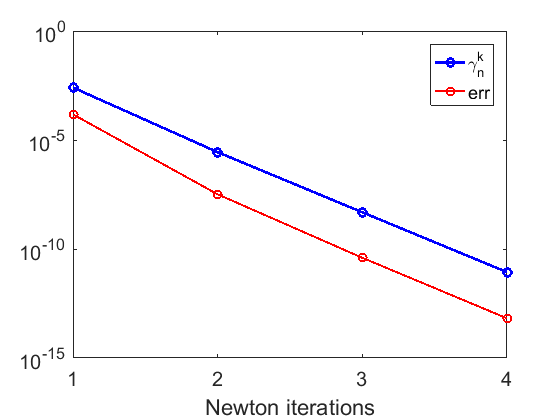}
	}
\caption{Assessment of the \textit{a posteriori} indicator $\gamma_k^n$ for the linearization error}%
\label{fig:newton}%
\end{figure}
We observe that for each $k = 1,\ldots,K$ the estimator is above the error, and they decrease with the same rate.
\section{Conclusions}
We considered the numerical approximation of the monodomain model, a system of a parabolic semilinear reaction-diffusion equation coupled with a nonlinear ordinary differential equation. The monodomain model arises from the (simplified) mathematical description of the electrical activity of the heart. In particular, we derived \textit{a posteriori} error estimators accounting for different sources of error (space/time discretization and linearization). Moreover, after obtaining an \textit{a priori} error estimate, we showed  reliability and efficiency (this latter under a suitable assumption) of the error indicators. Lastly, a set of numerical experiments assess the validity of the theoretical results.

\newpage
\printbibliography
\end{document}